\documentclass[11pt,reqno]{amsart}
\usepackage{amssymb}
\usepackage{amsmath}
\usepackage{graphicx,color}
\usepackage{epstopdf}
\usepackage{float}
\usepackage[bookmarks=false,unicode,colorlinks,urlcolor=red,citecolor=blue,linkcolor=blue]
{hyperref}

\usepackage{graphicx}

\usepackage{enumitem}
\usepackage{aliascnt}
\usepackage[margin=1in]{geometry}

\newtheorem{thm}{Theorem}[section]

\newaliascnt{lem}{thm}
\newtheorem{lem}[lem]{Lemma}
\aliascntresetthe{lem}

\newaliascnt{prop}{thm}
\newtheorem{prop}[prop]{Proposition}
\aliascntresetthe{prop}

\newaliascnt{cor}{thm}
\newtheorem{cor}[cor]{Corollary}
\aliascntresetthe{cor}

\newaliascnt{conj}{thm}

\aliascntresetthe{conj}

\newaliascnt{rmk}{thm}
\newtheorem{rmk}[rmk]{Remark}
\aliascntresetthe{rmk}

\newaliascnt{defn}{thm}
\newtheorem{defn}[defn]{Definition}
\aliascntresetthe{defn}

\newcommand\be{\begin{equation}}
\newcommand\en{\end{equation}}
\newcommand\pr{\begin{proof}}
\newcommand\epr{\end{proof}}

\newcommand{\R}{{\mathbb R}}

\numberwithin{equation}{section}

\begin{document}
\title[A parabolic-hyperbolic system modeling]{A parabolic-hyperbolic system modeling the growth of a tumor}
\begin{abstract}
  In this paper, we consider a model with tumor microenvironment involving nutrient density, extracellular matrix and matrix degrading enzymes, which satisfy a coupled system of PDEs with a free boundary.  For this coupled parabolic-hyperbolic free boundary problem, we prove that there is a unique radially symmetric solution globally in time. The stationary problem involves a ODE system which is transformed into a singular integro-differential equation.  We establish a well-posed theorem for such general types of equations by the shooting method; the theorem is then applied to our problem for the existence of a stationary solution. In addition, for this highly nonlinear problem, we also prove the uniqueness of the stationary solution, which is a nontrivial result. In addition, numerical simulations indicate that the stationary solution is likely  locally asymptotically stable for reasonable range of parameters.
  
\end{abstract}
\author{Rui Li}
\address{School of Mathematics, Renmin University of China, Beijing, 100872, P. R. China}
\email{liruicxis@ruc.edu.cn}
\author{Bei Hu}
\address{Department of Applied Computational Mathematics and Statistics, University of Notre Dame, Notre Dame, IN 46556, USA}
\email{b1hu@nd.edu}
\thanks{Corresponding author: Rui Li}
\thanks{keywords: Free boundary, Shooting method, Singular integro-differential equations, Uniqueness, Simulation}
\thanks{2010 Mathematics Subject Classification: 35R35, 35M10, 35L45, 34B16}
\maketitle


\section{Introduction}
It is estimated that there are 8.2 million cancer-related deaths worldwide every year. Tumor malignancy and metastatic progression are the primary cause, which leads to 90 percent of deaths from cancer.
Many recent cancer-related studies have pointed out that the remodeling of collagen fibers in the extracellular matrix (ECM) of the tumor microenvironment facilitates the migration of cancer cells during metastasis, since such modifications of ECM collagen fibers result in changes of ECM physical and biomachanical properties that affect cancer cell migration through the ECM \cite{MMM}. The ECM is defined as the diverse collection of proteins and sugars that surrounds cells in all solid tissues. This tissue compartment provides structural support by maintaining an insoluble scaffold, and this in turn defines the characteristic shape and dimensions of organs and complex tissues \cite{ECM}. Actually, various types of fibrous proteins are present in the ECM including collagens, elastins and laminis; among these, collagen is the most abound-ant and the main structural protein in the ECM \cite{FC}. In general, the ECM degradation caused by enzyme matrix is a key procedure for the ECM remodeling. In this paper, we try to use the matrix degrading enzymes (MDE) to describe the degrading process.

Over the last few decades, mathematical modeling has played a vital role in testing hypotheses, simulating the dynamics of complex systems and understanding the mechanistic underpinnings of dynamical systems. In particular, an increasing number of mathematical models describing solid tumor growth have been studied and developed; these models are classified into discrete cell-based models and continuum models. At the tissue level, continuum models provide a very good approximation. These models incorporate a system of partial differential equations (PDEs), where cell density, nutrients (i.e., oxygen and glucose), etc., are tracked. Modeling, mathematical analysis and numerical simulations were carried out in numerous papers, see \cite{AGLSC-JMB11,CuiFriedman-03,DMFriedman-NA86,DTGC-JTB04, EA-14,Friedman-04,G-72,PTTLC-BMB17,TLCW-13,Tao-JDE09, WFCMCL-JTB14,WZ-Non12} and the references therein. Lowengrub {\it et} {\it al} \cite{Lowengrub-Non10} provided a systematic review of tumor model studies.
However, in many of these models, the movement of the ECM within the tumor cells is ignored. Therefore, in order to better describe and understand the whole process and related mechanism, we study a mathematical model for the influence of the extracellular matrix (ECM) on tumor's evolution in terms of system of partial differential equation. This model basically consists of a system of parabolic equations and a hyperbolic equation for the density of the nutrient, for the matrix degrading enzymes (MDE) and the ECM concentration. Moreover, our model is more flexible, since it involves nonconstant coefficient $\mu(E)$ and allows the movement of the ECM fibres. All these considerations make our  model into a more reasonable and realistic setting, but lead to a more challenging
problem to analyze.

\section{The model}
In this section, we consider a PDE system to describe the evolution of the tumor. 

\subsection{Nutrients}

\par\vspace*{0.5ex}\par\noindent

Let $\Omega(t)$ denote the tumor domain at time $t$, and nutrient $\sigma$ within the tumor is modeled by a diffusion equation
\begin{equation}\label{LRsigma100}
c\frac{\partial\sigma}{\partial t}=\Delta\sigma + \gamma(\sigma_{B}-\sigma)-\lambda\sigma \mbox{ in }\Omega(t),
\end{equation}
where $c=T_{\text{diffusion}}/T_{\text{growth}}$ is the ratio of the nutrient diffusion time scale to the tumor growth (e.g. tumor doubling) time scale, $\gamma(\sigma_{B}-\sigma)$ denotes the nutrient supplied by the vasculature with $\gamma$ being the transfer rate of nutrient in blood to tissue and $\sigma_{B}$ being the concentration of nutrients in the vasculature, $\lambda\sigma$ describes the rate of consumption by the tumor. By appropriate change of variables, \eqref{LRsigma100} is reduced to (see \cite{FriedmanLam-JDE15})
\begin{equation}\label{LRsigm}
c\frac{\partial\sigma}{\partial t}=\Delta\sigma-\lambda\sigma \mbox{ in }\Omega (t).
\end{equation}

\subsection{Extracellular Matrix}
\par\vspace*{0.5ex}\par\noindent

The concentration of the ECM in the system is governed by contributions from three factors: haptotaxis, degrading, 
production. Here, there is a basic assumption that an equilibrium amount of nutrient $\bar{\sigma}$ is needed for tumor to sustain itself; beyond this $\bar{\sigma}$ the tumor grows, and below $\bar{\sigma}$ the tumor shrinks. Therefore a linear approximation for the proliferation $S$ is given by
\begin{equation}\label{LRsi}
S=\mu(\sigma-\bar{\sigma})\hspace{3em}(\bar{\sigma}>0),
\end{equation}
where $\mu\sigma$ represents the growth rate and $\mu\bar{\sigma}$ represents the death rate from apoptosis.
We shall employ Darcy's law (see \cite{CuiFriedman-03,DMFriedman-NA86,Friedman-04}):
\begin{equation}
\overrightarrow{V}=-\tilde{\mu}\nabla{p},
\end{equation}
where $\overrightarrow{V}$ represents the velocity of proliferating cells and $p$ the pressure within the tumor resulting from this proliferation. It is well known that Darcy's law describes the velocity of fluid in a porous medium, with the coefficient $\tilde{\mu}$ depending on the density of the porous medium, representing a mobility that reflects the combined effects of cell-cell and cell-matrix adhesion. In employing Darcy's law, we have assumed ECM to be the porous medium; $\tilde{\mu}=\tilde{\mu}(E)$ depends on the amount of ECM present in the tumor. By conservation of mass
\begin{equation}\label{LRVb}
\mbox{div}\overrightarrow{V}=S.
\end{equation}
Substituting \eqref{LRsi} into \eqref{LRVb}, we obtain
\begin{equation}\label{LRVc}
\mbox{div}\overrightarrow{V}=\mu(\sigma-\bar{\sigma}).
\end{equation}

In papers (\cite{CuiFriedman-03,DMFriedman-NA86,Friedman-04,FriedmanLam-JDE15}) both $\mu$ and $\tilde{\mu}$ are assumed to be constants; these are good approximations when ECM dose not vary much. Here, we shall incorporate a more reasonable assumption that both $\mu$ and $\tilde{\mu}$ also depend on ECM density $E$. It is clear that $\mu(E)$ and $\tilde{\mu}(E)$ are both monotone decreasing functions bounded from above and below by positive constants. In order to do so, we also need to incorporate the equation for $E$ (see \cite{DTGC-JTB04,EACR-MBE06,GC-JTB08}):

\begin{equation}\label{LRE}
\frac{\partial E}{\partial t}+\mbox{ div}(E\cdotp\overrightarrow{V}) = -\gamma m E+\phi(E),
\end{equation}
where the term $\mbox{div}(E\cdotp\overrightarrow{V})$ represents the movement of ECM owing to the cell proliferation $\overrightarrow{V}$; the term $-\gamma m E$ represents the degrading of ECM by MDE, here $m$ represents the concentration of MDE (see \cite{TLCW-13}); finally the term $\phi(E)$ is a positive term representing reorganization of ECM. Since the growth rate of ECM is smaller when the ECM is denser, $\phi(E)$ is a positive monotone decreasing function of $E$.

\subsection{Matrix degrading enzymes}
\par\vspace*{0.5ex}\par\noindent

MDE is produced by the tumor to degrade ECM so that the cells can escape. The equation for MDE is given by (see \cite{ACNST-JTM00,DTGC-JTB04})
\begin{equation}\label{LRm}
\frac{\partial m}{\partial t}=D_{m}\Delta m+\alpha-\beta m \mbox{ in } \Omega(t),
\end{equation}
where $\Delta m$ represents diffusion, $D_{m}$ is the constant diffusion coefficient and $-\beta m$ represents natural decay. Here we assume $\alpha$ to be a constant production rate by the tumor.

To summarise, the model studied in this paper is as follows:
\begin{subequations}\label{lirui01}\begin{align}
	\label{lirui01Sa} &c\frac{\partial\sigma}{\partial t}=\Delta\sigma-\lambda\sigma \hspace{-8em}&\mbox{ in }\Omega(t),\\
	\label{lirui01Ea} &\frac{\partial E}{\partial t}+\mathrm{div}(E\cdot\overrightarrow{V}) = -\gamma m E+\phi(E) \hspace{-8em}&\mbox{ in }\Omega(t),\\
	\label{lirui01ma} &\frac{\partial m}{\partial t} =D_{m}\triangle m+\alpha-\beta m \hspace{-8em}&\mbox{ in }\Omega(t),\\
	\label{lirui01pa} &\overrightarrow{V}=-\tilde{\mu}(E)\nabla p \hspace{-8em}&\mbox{ in }\Omega(t),\\
	\label{lirui01VA} &\mathrm{div}\overrightarrow{V} = \mu(E)(\sigma-\bar{\sigma}) \hspace{-8em}&\mbox{ in }\Omega(t).
	\end{align}\end{subequations}

\subsection{Boundary and initial conditions}
\par\vspace*{0.5ex}\par\noindent

We impose boundary conditions
	\begin{align}
	\label{lirui02S} \sigma =1 &\mbox{ on } \partial\Omega(t),\\
	\label{lirui02m} \frac{\partial m}{\partial n} =0 &\mbox{ on } \partial\Omega(t),\\
	\label{lirui02p} p =\kappa  &\mbox{ on } \partial\Omega(t).
	\end{align}
	Equation \eqref{lirui02S} represents a condition that the tumor is immersed in an environment of constant nutrients; equation \eqref{lirui02m} represents no exchange of MDE on the tumor boundary; and equation \eqref{lirui02p} represents the cell-to-cell adhesiveness, where $\kappa$ is the mean curvature. Finally, assuming the velocity is continuous up to the boundary, then
	\begin{equation}\label{lirui03}
	V_n=\overrightarrow{V}\cdotp\overrightarrow{n}=-\tilde{\mu}\nabla p\cdot\overrightarrow{n} \mbox{ on } \partial\Omega(t),
	\end{equation}
	where ${V_n}$ represents the velocity of the boundary $\partial\Omega(t)$ in the normal direction.
	
	Initial conditions:
	\begin{equation}\label{LRLR23c}
	\Omega(0)=\Omega_{0},\hspace{0.5em} \sigma|_{t=0}=\sigma_{0}(x),\hspace{0.5em}
	m|_{t=0}=m_{0}(x),\hspace{0.5em}
	E|_{t=0}=E_{0}(x).\hspace{0.5em}
\end{equation}

In comparison with a system assuming ECM to be constants, our system  is more reasonable and complex because we assume that ECM satisfies a hyperbolic equation coupled with nutrient $\sigma$ and pressure $p$ deriving from  cells' proliferation. In this paper we shall study the radially symmetric case. While tumors in vivo are not spherical, tumors in vitro are  typically of spherical shape \cite{WBL}. The structure of this paper is as follows. In section 3, we proceed to derive estimates to establish global existence and uniqueness and gave the lower bounds estimate of tumor radius $R(t)$. In section 4, we prove that there exists a unique stationary solution by the shooting method. In section 5, the corresponding numerical simulation confirms the expected asymptotic stability in certain parameter range. In appendix,  we proved that the well-posed theorem for the general singular integro-differential equation, which is a preliminary work for section 4.
		
\section{ Time dependent solution}

In this section we are concerned with the existence of radially symmetric solution. 

\subsection{Reformulation of the radially symmetric problem}
In order to prove the existence of the solution, for convenience, we do a reformulation for the radially symmetric problem.

  In radially symmetric case $\overrightarrow{V}=\frac{x}{|x|}u$, we have by \eqref{LRVc}
$$\begin{aligned}\mbox{div}
\overrightarrow{V}&=u_{r}+\frac{2}{r}u=\mu(E)(\sigma-\bar{\sigma}),\\
\mbox{div}(E\overrightarrow{V})&= E\cdot\mbox{div}\overrightarrow{V} + (\nabla E)\cdotp\overrightarrow{V}=E\mu(E)(\sigma-\bar{\sigma})+ u E_{r}.
\end{aligned}$$
Substituting this into \eqref{LRE}, we obtain
\begin{equation}\label{LRLR21}\begin{aligned}
\frac{\partial E}{\partial t}+u\frac{\partial E}{\partial r} = Q(\sigma, m, E),
\end{aligned}\end{equation}
where
\begin{equation}
Q(\sigma, m, E) = -\gamma m E+\phi(E)-E\mu(E)(\sigma-\bar{\sigma}).
\end{equation}
Furthermore, in the radially symmetric case, the mean curvature $\kappa$ is a constant, and once $\sigma$ and $E$ are determined, one can uniquely solve $p$ from the following linear elliptic equation:
$$\begin{cases}
-\mbox{div} (\mu_{2}(E))\nabla p)=\mu_{1}(E)(\sigma-\bar{\sigma})&\mbox{ for }|x|<R(t),\\
p=\kappa=\frac{1}{R(t)}&\mbox{ for }|x|=R(t).
\end{cases}$$
Therefore, we can drop the equation for $p$. In summary, the equations in the radially symmetric case are:
\begin{align}\label{meq}
c\frac{\partial\sigma}{\partial t}=\frac{1}{r^{2}}\frac{\partial}{\partial r} \left(r^{2}\frac{\partial\sigma}{\partial r}\right)-\lambda\sigma, \hspace{1em} 0\leq r< R(t), t>0,\\
\frac{\partial E}{\partial t}+u\frac{\partial E}{\partial r}=Q(\sigma, m, E),\hspace{1em}0\leq r< R(t), t>0,\\
\frac{\partial m}{\partial t} =D_m\frac{1}{r^{2}}\frac{\partial}{\partial r}\left(r^{2}\frac{\partial m}{\partial r}\right)+\alpha-\beta m,\hspace{1em} 0\leq r< R(t), t>0,\\
u_{r}+\frac{2}{r} u=\mu(E)(\sigma-\bar{\sigma}),\hspace{1em}0\leq r< R(t), t>0.
\end{align}
The system is supplemented with boundary conditions
\begin{align}
\sigma(R(t),t)=1,\hspace{1em} t>0,\\
\frac{\partial m}{\partial r}(R(t), t)=0,\hspace{1em} t>0,
\end{align}
and free boundary condition (assuming continuity of velocity up to the boundary)
\begin{equation}\label{diR}
\frac{dR(t)}{dt}=u(R(t),t).
\end{equation}
The system is also supplemented with initial conditions
\begin{align}
\sigma(r,0)=\sigma_{0}(r),\hspace{1em} 0\leq r\leq R(0),\\
m(r,0)=m_{0}(r),\hspace{1em} 0\leq r\leq R(0),\\
E(r,0)=E_{0}(r),\hspace{1em} 0\leq r\leq R(0).
\end{align}

By symmetry,
\begin{align}
\frac{\partial\sigma}{\partial r}(0,t)=0,\hspace{2ex} \frac{\partial E}{\partial r}(0,t)=0,\hspace{2ex}
\frac{\partial m}{\partial r}(0,t)=0,\hspace{2ex} u(0,t)=0.
\end{align}

We now use a change of variables that transform the free boundary into a fixed boundary:
$$\frac{r}{R(t)}\rightarrow\tilde{r}, t\rightarrow\tilde{t},\sigma\rightarrow\tilde{\sigma}, m\rightarrow\tilde{m}, E\rightarrow\tilde{E},\frac{u}{R}\rightarrow\tilde{u}.$$

For simplicity, we drop "$\sim$" in our notation, and then $\sigma(r,t),$ $m(r,t),$ $E(r,t),$ $u(r,t)$ satisfy
\begin{subequations}\label{LRLR23}\begin{align}
	\label{LRLR23Sa} &c\sigma_{t}=\frac{1}{[R(t)]^{2}}(\sigma_{rr}+\frac{2}{r}\sigma_{r}) + \frac{R'(t) }{R(t)}r\sigma_{r}-\lambda\sigma,\hspace{1em}0<r<1, t>0,\\
	\label{LRLR23ma} &m_{t}=D_m\frac{1}{[R(t)]^{2}}(m_{rr}+\frac{2}{r}m_{r}) + \frac{R'(t)}{R(t)}rm_{r} +\alpha-\beta m,\hspace{1em}0<r<1, t>0,\\
	\label{LRLR23Ea} &E_{t} + vE_{r}=Q(\sigma,m,E),\hspace{1em}0<r<1, t>0,\\
	\label{LRLR23v}  &v(r,t)=u(r,t)-ru(1,t),\hspace{1em}0<r<1, t>0,\\
	\label{LRLR23ua} &u_{r}+\frac{2}{r}u=\mu(E)(\sigma-\bar{\sigma}),\hspace{1em}0<r<1, t>0,\\
	\label{LRLR23R}  &R_{t}(t)=R(t){u}(1,t),\hspace{2ex}t>0,\\
	\label{LRLR23Sb} &\sigma_{r}(0,t)=0,\hspace{2ex}\sigma(1,t)=1,\hspace{2ex}t>0,\\
	\label{LRLR23mb} &m_{r}(0,t)=0,\hspace{2ex}m_{r}(1,t)=0,\hspace{2ex}t>0,\\
	\label{LRLR23ub} &u(0,t)=0,\hspace{2ex} t>0,
	\end{align}\end{subequations}
with initial conditions
\begin{subequations}\label{LRLR23c}\begin{align}
	\label{LRLR23Rc} &R(0)=R_{0},\\
	\label{LRLR23Sc} &\sigma(r,0)=\sigma_{0}(r),\hspace{1em}0<r<1,\\
	\label{LRLR23mc} &m(r,0)=m_{0}(r),\hspace{1em}0<r<1,\\
	\label{LRLR23Ec} &E(r,0)=E_{0}(r),\hspace{1em}0<r<1,
	\end{align}\end{subequations}
where $\sigma_{0}, m_{0}\in C^{2}[0,1]$ and $E_{0}\in C^{1}[0,1]$ satisfy
\begin{equation}\label{LRLR23ini} R_{0}>0,\sigma_{0r}(0)=0,\sigma_{0}(1)=1, m_{0r}(0)=0, m_{0r}(1)=0, E_{0r}(0)=0.\end{equation}

In particular, from \eqref{LRLR23ub} and \eqref{LRLR23v}, we see that
$$v(0,t)=0,\hspace{1em}v(1,t)=0.$$
This implies that no boundary conditions are needed for $E$ at $r=0,1$.

Throughout this section we assume that the initial data of \eqref{LRLR23c} satisfy the following assumption: $\sigma_{0}, m_{0}$ are radially symmetric functions and belong to $W^{2,p}(B_{1})$ (for some fixed $p>5$), $E_{0}\in C^{1}[0,1]$, $R_{0}>0$. By biological consideration, these initial functions are nonnegative and do not vanish completely.

\subsection{Local existence and uniqueness}
 We start with local existence and uniqueness by applying the contraction mapping principle.

For $T>0$, we set
$$Q_{T}=(0,1)\times(0,T),\hspace{2ex}\bar{Q}_{T}=[0,1]\times[0,T].$$

\begin{defn}\label{defn:XT}
	For any given $T>0$, we define a complete metric space $(X_{T},d)$ as follows:
	$X_{T}$ is the subset of $W^{1,p}(0,T)\times C^{0,0}(\bar{Q}_{T})$ consisting of a collection of pairs of functions $R=R(t)$, $E=E(r,t)$ satisfying
	
	(i) $R\in W^{1,p}(0,T)$, $R(0)=R_{0}$ and
	\begin{equation}\label{LRLR21R}
	\|R'\|_{L^{p}(0,T)}\leq 1,\hspace{2ex}
	\frac{R_{0}}{2}\leq R(t)\leq 2R_{0},\hspace{2ex}t\in[0,T].
	\end{equation}
	
	(ii) $E\in C^{0,0}(\bar{Q}_{T})$, $E(r,0)=E_{0}(r)$ and
	\begin{equation}\label{LRLR21E}\|E\|_{C^{0,0}(\bar{Q}_{T})}\leq M_{1},\end{equation}
	where $M_{1}=\|E_{0}\|_{L^{\infty}}+1$.
	
	The metric $d$ in $X_{T}$ is as follows
	$$\begin{aligned}
	& d((R_{1},E_{1}), (R_{2},E_{2}))\\
	=& \|R_{1}-R_{2}\|_{C[0,T]} + \|R'_{1}-R'_{2}\|_{L^{p}(0,T)} + \|E_{1}-E_{2}\|_{C^{0,0}(\bar{Q}_{T})}.
	\end{aligned}$$
\end{defn}

Given a pair $(R, E)\in X_{T}$, we define $\sigma=\sigma(r,t)$ as the solution of the initial boundary value problem \eqref{LRLR23Sa}, \eqref{LRLR23Sb}, \eqref{LRLR23Sc}, and $m=m(r,t)$ as the solution of the initial boundary value problem \eqref{LRLR23ma}, \eqref{LRLR23mb}, \eqref{LRLR23mc}.
Let us define $u=u(r,t)$ as the solution of \eqref{LRLR23ua}, \eqref{LRLR23ub} and $v=v(r,t)$ by \eqref{LRLR23v}, that is
\begin{equation}\label{LRLR24u}
u(r,t)=\frac{1}{r^{2}}\int_{0}^{r}g(\sigma(\rho,t), E(\rho,t))\rho^{2}d\rho,
\end{equation}
where $g(\sigma, E)=\mu(E)({\sigma}-\bar{\sigma})$.
We next define $(\hat{R},\hat{E})=\mathcal{F}(R, E)$ such that
\begin{align}
\label{LRLR24Ea}\hat{E}_{t}(r,t) + v(r,t)\hat{E}_{r}(r,t)&=Q({\sigma}, m, \hat{E}), \hspace{1em}0<r<1, t>0,\\
\label{LRLR24Eb}\hat{E}(r,0)&=E_{0}(r),\hspace{1em}0<r<1,\\
\label{LRLR24Ra}\hat{R}'&=u(1,t)\hat{R}(t),\hspace{1em} t>0,\\
\label{LRLR24Rb}\hat{R}(0)&=R_{0}.
\end{align}

We shall prove the existence of a local solution of \eqref{LRLR23}-\eqref{LRLR23ini} by using the contraction mapping theorem for a map $\mathcal{F}:X_{T}\rightarrow X_{T}$.

Clearly, the problem \eqref{LRLR24Ra}-\eqref{LRLR24Rb} can be solved explicitly
\begin{equation}\label{LRLR24R}
\hat{R}(t)=R_{0}\exp\left(\int_{0}^{t}u(1,\tau)d\tau\right).
\end{equation}

To uniquely solve \eqref{LRLR24Ea}-\eqref{LRLR24Eb},
we introduce the characteristic curves ending at $(r,t)$
\begin{equation}\label{LRLR25}\begin{cases}
\displaystyle\frac{d\xi(r,t;s)}{ds}=v(\xi(r,t;s),s),\\
\xi(r,t;t)=r.
\end{cases}\end{equation}
Since $v(r,t)$ is continuous in $(r,t)$ and Lipschitz in $r$, the characteristic curves $\xi$ is well defined for $0\leq s\leq t.$
We then rewrite \eqref{LRLR24Ea} in the form
$$\frac{d}{ds}\hat{E}(\xi(r,t;s),s)=Q({\sigma}(\xi(r,t;s),s),m(\xi(r,t;s),s),\hat{E}(\xi(r,t; s),s)).$$
Note that $v$ satisfies \eqref{LRLR23v}, we see that the characteristic curves do not leave and enter the space interval $(0,1)$.
For simplification of notation, we denote $\mathcal{Q}(r,t,E)=Q(\sigma(r,t), m(r,t), E)$ and consider
\begin{equation}\label{LRLR26}\begin{cases}
\displaystyle\frac{d\tilde{E}(r,t;s)}{ds} = \mathcal{Q}(\xi(r,t;s),s, \tilde{E}(r,t;s)), \hspace{1em}0<s<t,\\
\tilde{E}(r,t;0)=E_{0}(\xi(r,t;0)).
\end{cases}\end{equation}
Clearly, \eqref{LRLR26} admits a unique (local) solution $\tilde{E}$.
Thus, $\hat{E}(r,t)=\tilde {E}(r,t;t)$ is the solution of \eqref{LRLR24Ea}-\eqref{LRLR24Eb}.

If we regard equation \eqref{LRLR23Sa} as a 1-dimensional parabolic equation
with the spatial variable $r$, then the coefficient of $\partial\sigma/\partial r$
has singularity at tumor center $r=0$ due to
$$\Delta\sigma=\frac{\partial^{2}\sigma}{\partial r^{2}} + \frac{2}{r}\frac{\partial\sigma}{\partial r}.$$
However, this singularity can be eliminated by employing the three-dimensional Cartesian coordinate.

Due to the assumption imposed on $R(t)$ in \eqref{LRLR21R}, we see that the coefficient $R'/R$ in equation \eqref {LRLR23Sa} and equation \eqref{LRLR23ma} only belongs to $L^{p}$. One can apply the classical parabolic theory to obtain the strong solution $\sigma$ and $m$ exist and belong to $W^{2,1,p}(Q_{T})$, see Theorem 9.1 and its corollary of chapter IV in \cite{Lady-68}.

In order to prove $\mathcal{F}$ maps $X_{T}$ into itself for some small $T$, it suffices to estimate the norms of $(\hat{R},\hat{E})$ as well as $\sigma,m,u,v$.

For notational convenience, in the sequel we shall denote by $C$ any one of several constants which depend on $R_{0}$, $M_1$, but does not depend on $T\in(0,1)$; we shall not keep track of their special forms since this will have no bearing on future considerations.

\begin{lem}\label{lem:102sigma}
	If $R(t)$ satisfies \eqref{LRLR21R}, then the strong solutions $\sigma$ and $m$ admit the following uniform bounds
	\begin{equation}\label{LRLR29sigma}\begin{aligned}
	\|\sigma\|_{W^{2,1,p}(B_{1}\times(0,T))}&\leq C,
	\|\sigma\|_{C^{1+\alpha,(1+\alpha)/2}(\bar{B}_{1}\times[0,T])}&\leq C,
	\end{aligned}\end{equation}
	\begin{equation}\label{LRLR29m}\begin{aligned}
	\|m\|_{W^{2,1,p}(B_{1}\times(0,T))}&\leq C,
	\|m\|_{C^{1+\alpha,(1+\alpha)/2}(\bar{B}_{1}\times[0,T])}&\leq C,
	\end{aligned}\end{equation}
	where $\alpha=1-5/p$ and $B_{1}$ is the unit ball in $\R^{3}$.
\end{lem}
\begin{proof}
	The bounds for $\sigma$ and $m$ are similar, we focus the uniform estimates for $\sigma$ only.
	Note that all functions are defined in the time interval $[0,T]$, we can extend $R(t)$, so that it is defined in a fixed interval [0,1]. More precisely,
	$\tilde{R}\in W^{1,p}(0,1)$ is defined as follows
	$$\tilde{R}(t)=\begin{cases}
	R(t),\hspace{0.5ex}\mbox{for}\hspace{0.5ex} t\in[0,T],\\
	R(T),\hspace{0.5ex}\mbox{for}\hspace{0.5ex} t\in(T,1].
	\end{cases}$$
	It is clear that $\tilde{R}'(t)\equiv 0$ for $T<t<1.$
	Clearly,
	$$\|\tilde{R}\|_{W^{1,p}(0,1)}\leq\|\tilde{R}\|_{C[0,1]}+\|\tilde{R}'\|_{L^{p}(0,1)}\leq 2R_{0}+1.$$
	Since the embedding $W^{1,p}(0,1)\hookrightarrow C^{\alpha}[0,1]$ is continuous, we conclude that $1/\tilde{R}^{2}$ is continuous and
	$$\|\tilde{R}\|_{C^{\alpha}[0,1]}\leq C\|\tilde{R}\|_{W^{1,p}(0,1)}\leq C.$$
	We now define $\tilde{\sigma}$ to be the solution of \eqref{LRLR23Sa}, \eqref{LRLR23Sb}, \eqref{LRLR23Sc} (with $R$ replaced by $\tilde{R}$) in the time interval $[0,1]$.
	From $L^{p}$ theory \cite{Lady-68}, we see that $\tilde{\sigma}\in W^{2,1,p}(B_{1}\times(0,T))$ exists and
	\begin{equation}\label{LRLR28a}
	\|\tilde{\sigma}\|_{W^{2,1,p}(B_{1}\times(0,1))}\leq C.
	\end{equation}
	Recalling that the embedding
	$$W^{2,1,p}(B_{1}\times(0,1))\hookrightarrow C^{1+\alpha,(1+\alpha)/2}(\bar{B}_{1}\times[0,1])$$
	is continuous, it follows that
	\begin{equation}\label{LRLR28b}
	\|\tilde{\sigma}\|_{C^{1+\alpha,(1+\alpha)/2}(\bar{B}_{1}\times[0,1])}\leq C\|\tilde{\sigma}\|_{W^{2,1,p}(B_{1}\times(0,1))}\leq C.
	\end{equation}
	By uniqueness of parabolic equation, we see $\tilde{\sigma}$ and $\sigma$ are the same in the time interval $[0,T)$. The conclusion immediately follows from \eqref{LRLR28a} and \eqref{LRLR28b}.
\end{proof}

From \eqref{LRLR24u}, we deduce that
\begin{equation}\label{LRLR29u}\|u\|_{C^{1, 0}(\bar{Q}_{T})}\leq C,~~\|v\|_{C^{1, 0}(\bar{Q}_{T})}\leq C.\end{equation}
From \eqref{LRLR24R} and \eqref{LRLR24Ra}, we obtain that
\begin{equation}\label{LRLR29R} \|\hat{R}\|_{C^{1}[0,T]}\leq C. \end{equation}
In particular, for sufficiently small $T>0$
\begin{equation}\label{LRLR32R}\begin{aligned}
\max_{t\in[0,T]}|\hat{R}(t)-R_{0}|\leq CT\leq R_{0}/2,
|\hat{R}'(t)|_{L^{p}(0,T)}\leq CT^{1/p}\leq1.\end{aligned}\end{equation}
Therefore, $\hat{R}=\hat{R}(t)$ satisfies \eqref{LRLR21R} in the definition of $X_{T}$ provided $T$ is sufficiently small.

\begin{lem}\label{lem:104E}
	For any sufficiently small $T>0$, if $(R,E)\in X_{T}$, then the unique solution $\hat{E}$ of \eqref{LRLR24Ea}-\eqref{LRLR24Eb} is well-defined for $t\in[0,T]$ and satisfies
	\begin{equation}\label{LRLR29E}
	\|\hat{E}\|_{C^{1,1}(\bar{Q}_{T})}\leq C.
	\end{equation}
\end{lem}

\begin{proof}
	Note that $0\leq\tilde{E}(r,t;0)\leq\|E_{0}\|\leq M_1$ and
	$$\mathcal{Q}(r,t,0)\geq0,~~\mathcal{Q}(r,t,E)\leq\Psi(E)$$
	for $r\in[0,1]$, $t\in(0,T]$, $E\geq0$, where $\Psi(E)$ is some positive smooth function defined in $[0,\infty)$.
	Let $y(s)$ denote the unique solution of the initial problem
	$$\frac{dy}{ds}=\Psi(y),\hspace{2ex} y(0)=M_{0}$$
	which exists at least in a finite interval $[0,h]$ for some $h>0$.
	Then by the comparison principle for ODE, we deduce that the solution $\tilde{E}$ of \eqref{LRLR26} satisfies
	$$0\leq\tilde{E}(r,t;s)\leq y(s),$$
	$s\in [0,h]$, where $\tilde{E}$ is a solution of \eqref{LRLR26}. This shows that $\tilde{E}$ exists and is bounded in $[0,1]\times[0,h]\times[0,h]$ provided $T>h$. In particular, $\hat{E}$ exist and is bounded in $\bar{Q}_{T}$ for $T\leq h$.
	
	From \eqref{LRLR29u} and \eqref{LRLR23v},  we see that $\|v\|_{C^{1,0}}\leq C$.
	It follows that $\xi$, defined by \eqref{LRLR25},
	belongs to $C^{1,1,1}$ and satisfies
	$$\|\xi\|_{L^{\infty}}+\|\xi_{r}\|_{L^{\infty}}
	+\|\xi_{t}\|_{L^{\infty}}+\|\xi_{s}\|_{L^{\infty}}\leq C.$$
	Recalling that $\mathcal{Q}(r,t,E)$ is $C^{1+\alpha}$ in $r$, $C^{(1+\alpha)/2}$ in $t$ and smooth in $\tilde{E}$,
	we deduce from \eqref{LRLR26} that $|\tilde{E}_{r}|<C$ and hence
	$$|\hat{E}_{r}|<C\mbox{ in }\bar{Q}_{T},$$
	here we write $\hat{E}_{r}(r,t)=\tilde{E}_{r}(r,t;t)$. Finally, we derive from equation \eqref{LRLR23Ea} that $\hat{E}_{t}(r,t)=- v(r,t)\hat{E}_{r}(r,t)+Q(\sigma,m,\hat{E})$ is also bounded in $\bar{Q}_{T}$.
	This completes the proof.
\end{proof}

\begin{cor}\label{cor:105E}
	For sufficiently small $T>0$, if $(R,E)\in X_{T}$, then $\hat{E}$ satisfies \eqref{LRLR21E}.
\end{cor}
\begin{proof}
	From \autoref{lem:104E}, we have
	$$|\hat{E}(r,t)|\leq\|E_{0}\|_{L^{\infty}}+ CT.$$
	This shows that $\hat{E}=\hat{E}(r,t)$ satisfies \eqref{LRLR21E} provided $T$ is sufficiently small.
\end{proof}

Combing \eqref{LRLR32R} and \autoref{cor:105E}, we have established the following
\begin{prop}\label{prop:106}
	For sufficiently small $T>0$, $\mathcal{F}$ is well-defined and maps $X_{T}$ into itself.
\end{prop}

We now establish that $\mathcal{F}$ is a contraction mapping for sufficiently small $T$. Let $(R_{i},E_{i})\in X_{T}$ for $i=1, 2$. Let $\sigma_{i}$, $m_{i}$, $u_{i}$, $v_{i}$, $(\hat{R}_{i},\hat{E}_{i})=\mathcal{F}(R_{i},E_{i})$ be the solution corresponding to $(R_{i},E_{i})$.
Recall
\begin{equation}\begin{aligned}
&d=d((R_{1},E_{1}), (R_{2},E_{2}))\\
=&\|R_{1}-R_{2}\|_{C[0,T]}
+\|R'_{1}-R'_{2}\|_{L^{p}(0,T)}+\|E_{1}-E_{2}\|_{C^{0,0}(\bar{Q}_{T})}.
\end{aligned}\end{equation}
From equation \eqref{LRLR23Sa}, we see $\sigma_{*}=\sigma_{1}-\sigma_{2}$ satisfies
$$\begin{cases}
c\frac{\partial\sigma^{*}}{\partial t} = \frac{1}{R_{1}^{2}(t)}(\frac{\partial^{2}\sigma^{*}}{\partial r^{2}} + \frac{2}{r}\frac{\partial\sigma^{*}}{\partial r}) + \frac{R'_{1}(t) }{R_{1}(t)}r \frac{\partial\sigma^{*}}{\partial r}-\lambda\sigma^{*}+h(r,t),\\
(\sigma^{*})_{r}(0,t)=0,\sigma^{*}(1,t)=0,\\
\sigma^{*}(r,0)=0,
\end{cases}$$
where
$$h(r,t)=(\frac{1}{R_{1}^{2}}-\frac{1}{R_{2}^{2}})(\frac{\partial^{2}\sigma_{2}}{\partial r^{2}} + \frac{2}{r}\frac{\partial\sigma_{2}}{\partial r}) + (\frac{R'_{1}}{R_{1}}-\frac{R'_{2}}{R_{2}})r\frac{\partial\sigma_{2}}{\partial r}.$$
From the estimates of $\sigma$ (\autoref{lem:102sigma}), we derive
$$\|h\|_{L^{p}(B_{1}\times(0,T))}\leq C(\|R_{1}-R_{2}\|_{C[0,T]} +\|R'_{1}-R'_{2}\|_{L^{p}(0,T)}).$$
Using the same method as that in \autoref{lem:102sigma}, we deduce
\begin{equation}\label{LRLR34sigma}\begin{aligned}
&\|\sigma_{1}-\sigma_{2}\|_{W^{2,1,p}(B_{1}\times(0,T))}\leq Cd,\\
&\|\sigma_{1}-\sigma_{2}\|_{C^{1+\alpha,(1+\alpha)/2}(\bar{B}_{1}\times[0,T])}\leq Cd.
\end{aligned}\end{equation}
Similarly,
\begin{equation}\label{LRLR34m}\begin{aligned}
&\|m_{1}-m_{2}\|_{W^{2,1,p}(B_{1}\times(0,T)}\leq Cd,\\
&\|m_{1}-m_{2}\|_{C^{1+\alpha,(1+\alpha)/2}(\bar{B}_{1}\times[0,T])}\leq Cd.
\end{aligned}\end{equation}
By the definition of $u_{i}$ and $v_{i}$ we have
\begin{equation}\label{LRLR34u}\begin{aligned}
\|u_{1}-u_{2}\|_{C^{1,0}(\bar{Q}_{T})}&\leq C(\|\sigma_{1}-\sigma_{2}\|_{\infty}+\|E_{1}-E_{2}\|_{\infty})\leq Cd, \\ \|v_{1}-v_{2}\|_{C^{1,0}(\bar{Q}_{T})}&\leq Cd.
\end{aligned}\end{equation}
Set $R^{*}=\hat{R}_{1}-\hat{R}_{2}$. Then, by direct calculations we see that $R^{*}$ satisfies
\begin{equation}\label{LRLR36a}\begin{cases}
\frac{d R^{*}(t)}{dt}=\bar{g}_{1}(t)R^{*}(t) + \bar{g}_{2}(t),\\
R^{*}(0)=0,
\end{cases}\end{equation}
or
\begin{equation}\label{LRLR36c}
R^{*}(t)=
\int_{0}^{t}\bar{g}_{2}(\tau)e^{\int_{\tau}^{t}g_{1}(\tilde{\tau})d\tilde{\tau}}d\tau,
\end{equation}
where $\bar{g}_{1}(t)=u_{1}(1,t)$, $\bar{g}_{2}(t)=R_{2}(t)(u_{1}(1,t)-u_{2}(1,t))$ satisfy
\begin{equation}\label{LRLR36b}
|\bar{g}_{1}(t)|\leq C, |\bar{g}_{2}(t)|\leq Cd, t\in[0,T]
\end{equation}
by using \eqref{LRLR29u}, \eqref{LRLR29R} and \eqref{LRLR34u}.
One can easily derive from \eqref{LRLR36a}-\eqref{LRLR36b} that
\begin{equation}\label{LRLR34R}\begin{aligned}
\|\hat{R}_{1}-\hat{R}_{2}\|_{C[0,T]}&\leq CTd,\\
\|\hat{R}_{1}'-\hat{R}_{2}'\|_{C[0,T]}&\leq Cd,
\|\hat{R}_{1}'-\hat{R}_{2}'\|_{L^{p}(0,T)}&\leq CT^{1/p}d.
\end{aligned}\end{equation}

\begin{lem}\label{lem:107E}
	There holds
	\begin{equation}\label{LRLR34E}
	\|\hat{E_{1}}-\hat{E_{2}}\|_{C(\bar{Q}_{T})}\leq CTd
	\end{equation}
	for sufficiently small $T$.
\end{lem}

\begin{proof}
	Note that $\hat{E}^{*}=\hat{E}_{1}-\hat{E}_{2}$ satisfies
	\begin{equation}\label{LRLR35}\begin{cases}
	\displaystyle\frac{\partial\hat{E}^{*}}{\partial t}+\frac{\partial\hat{E}^{*}}{\partial r} v_{1}-A(r,t)\hat{E}^{*}=h(r,t),\\
	\hat{E}^{*}(r,0)=0,
	\end{cases}\end{equation}
	where
	$$A(r,t)=\int_{0}^{1}Q(\sigma_{2}(r,t), m_{2}(r,t),\theta\hat{E}_{1}(r,t)+(1-\theta)\hat{E}_{1}(r,t))d\theta,$$
	$$\begin{aligned}h(r,t)=&-(v_{1}-v_{2})\frac{\partial\hat{E}_{2}}{\partial r} + [Q(\sigma_{1}, m_{1}, \hat{E}_{1})-Q(\sigma_{2}, m_{1}, \hat{E}_{1})]\\
	&+[Q(\sigma_{2}, m_{1}, \hat{E}_{1})-Q(\sigma_{2}, m_{2}, \hat{E}_{1})].
	\end{aligned}$$
	Using the estimates \eqref{LRLR29sigma},\eqref{LRLR34sigma} for $\sigma_{i}$, the estimates \eqref{LRLR29m},\eqref{LRLR34m} for $m_{i}$, the estimates \eqref{LRLR29u},\eqref{LRLR34u} for $u_{i},v_{i}$ and the estimates \eqref{LRLR29E} for $\hat{E}_{i}$, we get
	$$|A(r,t)|\leq C, |h(r,t)|\leq Cd,~~ (r,t)\in\bar{Q}_{T}.$$
	Hence, integrating \eqref{LRLR35} along its characteristics as before, we find that
	$$\|\hat{E}^{*}\|_{C(\bar{Q}_{T})}\leq CTd.$$
	This finishes the proof.
\end{proof}

Now we show the local existence and uniqueness of solution to \eqref{LRLR23}-\eqref{LRLR23c}.

\begin{thm}\label{thm:108UniqueLocal}
	Assume that the initial data satisfy \eqref{LRLR23ini}.
	Then there exists a $T>0$ which only depends on $\|\sigma_{0}\|_{W^{2,p}(B_{1})}$,
	$\|m_{0}\|_{W^{2,p}(B_{1})}$, $R_{0}$, and $\|E_{0}\|_{C^{1}[0,1]}$,
	such that problem \eqref{LRLR23}-\eqref{LRLR23c} admits a unique solution
	$\sigma, m\in W^{2,1,p}(B_{1}\times(0,T))$, $u, v\in C^{1,0}(\bar{Q}_{T})$,
	$E\in C^{1,1}(\bar{Q}_{T})$, $R\in C^{1}[0,T]$.
\end{thm}
\begin{proof}
	From \autoref{prop:106}, \eqref{LRLR34R} and \autoref{lem:107E}, we deduce that for sufficiently small $T>0$, $\mathcal{F}$ is a contraction mapping from $X_{T}$ into itself. Therefore, there is a unique fixed point in $X_{T}$ for $\mathcal{F}$ , and thus \eqref{LRLR23}-\eqref{LRLR23c} admits a unique solution $(\sigma,m,E,u,v,R)$ in the time interval $[0,T]$.
\end{proof}

\begin{rmk}
	This theorem shows the local existence and uniqueness of radially symmetric solution of our equation.
	If the $E_{0}$ satisfies $(E_{0})_{r}(0)=0$, then one can prove that $E_{r}(0,t)=0$ for all $t$.
	
	If initial data $\sigma_{0}, m_{0}, E_{0}$ are radially symmetric function, and $\sigma_{0}, m_{0}\in C^{2,\alpha}(\bar{B}_{1})$ (for some $\alpha\in(0,1)$), $E_{0}\in C^{1}(\bar{B}_{1})$, $R_{0}>0$, then the standard PDE theory tells us the solution obtained in \autoref{thm:108UniqueLocal} are also radially symmetric  and are more regular, i.e. $\sigma, m\in C^{2+\alpha,(2+\alpha)/2}(\bar{B}_{1}\times[0,T])$, $u, v\in C^{1,0}(\bar{Q}_{T})$,
	$E\in C^{1}(\bar{B}_{1}\times[0,T])$, $R\in C^{1}[0,T]$.
\end{rmk}

\subsection{Global existence}

\begin{thm}\label{thm:109}
	Assume that the initial data satisfy \eqref{LRLR23ini}. Then problem \eqref{LRLR23}-\eqref{LRLR23c} admits a unique solution $(\sigma,m,u,v,E,R)$ globally in time.
\end{thm}

\begin{proof}
	Suppose to the contrary that $[0,\tilde{T})$ is the maximum time interval $(\tilde{T}<\infty)$ for the existence of the solution. We proceed to derive necessary estimates for the global existence.
	
	Employing the maximum principle for parabolic equations, we deduce that
	\begin{equation}\label{LRLR41}
	0<\sigma(r,t), m(r,t)\leq \max\{\|\sigma_{0}\|_{L^{\infty}},\frac{\alpha}{\beta},\|m_{0}\|_{L^{\infty}}\}
	\end{equation}
	for all $(r,t)\in[0,1]\times[0,\tilde{T})$. We shall denote by $C$  various constant which is independent of $\tilde{T}$, and by $C(\tilde{T})$  various constants which depends on $\tilde{T}\in(0,\infty)$.
	Since $\mu(E)$ is decreasing and bounded, we obtain from \eqref{LRLR24u} and \eqref{LRLR41} that
	\begin{equation}\label{LRLR42}
	|u(r,t)|\leq C,|u_{r}(r,t)|\leq C\texttt{ for all }(r,t)\in[0,1]\times[0,\tilde{T}).
	\end{equation}
	From equation \eqref{LRLR23R} and \eqref{LRLR42}, we derive that
	\begin{equation}\label{LRLR43}
	R_{0}e^{-C\tilde{T}}\leq R(t)\leq R_{0}e^{C\tilde{T}},~~|R'(t)|\leq CR_{0}e^{C\tilde{T}}\texttt{ for all }t\in[0,\tilde{T}).
	\end{equation}
	By the $L^{p}$ theory (see \cite{Lady-68} or \cite{Lieberman-96}) and Sobolev inequalities, we get
	\begin{equation}\label{LRLR44}\begin{aligned}
	&\|\sigma(r,t), m(r,t)\|_{W^{2,1,p}(B_{1}\times[0,\tilde{T}))}\leq C(\tilde{T}),\\
	&\|\sigma(r,t), m(r,t)\|_{C^{1+\alpha,\frac{1+\alpha}{2}}(\bar{B}_{1}\times[0,\tilde{T}])}\leq C(\tilde{T}).
	\end{aligned}\end{equation}
	From the bounds \eqref{LRLR41} for $\sigma$ and $m$ and the assumption on $Q$, we obtain that $\mathcal{Q}(r,t,E)=Q(\sigma(r,t),m(r,t),E)$ satisfies
	$$|\mathcal{Q}(r,t,E)|\leq C^{*}(E+1), r\in[0,1],t\in[0,\tilde{T}), E\geq0$$
	for some constant $C^{*}$ which is independent of $\tilde{T}\in(0,\infty)$.
	Note that $E(r,t)=\tilde{E}(r,t;t)$ and $\tilde{E}(r,t;s)$ satisfies
	\begin{equation}\begin{cases}
	\displaystyle\frac{d\tilde{E}(r,t;s)}{ds} = \mathcal{Q}(\xi(r,t;s),s,\tilde{E}(r,t;s)),\hspace{1em}0<s<t,\\
	\tilde{E}(r,t;0)=E_{0}(\xi(r,t;0)).
	\end{cases}\end{equation}
	We conclude that
	$$|\tilde{E}(r,t;s)|\leq (\|E_{0}\|_{C[0,1]}+1) e^{C^{*}\tilde{T}}-1, r\in[0,1],0\leq s\leq t<T.$$
	By using the bounds \eqref{LRLR42} for $u$ and then for $v$, we see that the character curves $\xi(r,t;s)$ and its derivatives $\xi_{r}(r,t;s)$ are bounded by some constant $C(\tilde{T})$. Combining this with \eqref{LRLR44}, one obtain
	$$|\tilde{E}_{r}(r,t;s)|\leq C(\tilde{T}), r\in[0,1],0\leq s\leq t<T.$$
	This gives the bounds for $E_{r}$ and hence for $E_{t}$ by equation \eqref{LRLR23Ea}. Therefore,
	\begin{equation}\label{LRLR44}|E(r,t)|+|E_{r}(r,t)|+|E_{t}(r,t)|\leq C(\tilde{T})\texttt{ for all }(r,t)\in[0,1]\times[0,\tilde{T}).\end{equation}
	
Taking $\tilde{T}-\epsilon$ (where $0<\epsilon<\tilde{T}$ is arbitrary) as a new initial time, then we can extend the solution to $Q_{(\tilde{T}-\epsilon)+\delta}$ for some small $\delta>0$ proceeding as in the proof of \autoref{thm:108UniqueLocal}. Furthermore, the
	proof of \autoref{thm:108UniqueLocal} shows that $\delta$ depends only on an upper bound of the data at time $\tilde{T}-\epsilon$ and the the lower bounds of $R(\tilde{T}-\epsilon)$.
	By a priori estimate \eqref{LRLR42}-\eqref{LRLR44}, we find that $\delta$ depends only on $\tilde{T}$ (but $\delta$ is independent of $\epsilon$), i.e., $\delta=\delta(\tilde{T})$.
	If we take $\epsilon<\delta(\tilde{T})$, then we get
	$$(\tilde{T}-\epsilon)+\delta>\tilde{T},$$
	which contradicts the assumption that $[0,\tilde{T})$ is the maximum time interval for the existence of the solution. Therefore, the maximum time interval for the existence of the solution is $[0,\infty)$.
\end{proof}
\subsection{ The lower bound of $R(t)$}

In this subsection, we consider the case of $\mu(E)$ is near $0$ and study the lower bound of $R(t)$.

From  \eqref{meq}, we get 
$$u(t)=\frac{1}{R^{2}(t)}\int _0^{R(t)} \mu(E)(\sigma-\bar{\sigma})\rho^{2}d\rho.$$
Combining with equation \eqref{diR}, we have 
\begin{equation}\label{RE}
R^{2}R'(t)=\int_0^{R(t)} \mu(E)(\sigma-\bar{\sigma})\rho^2 d\rho.\end{equation}

From the maximum principle for parabolic equation, we see that
$$0<\sigma(r,t)<\hat{\sigma}\hspace{0.5ex}\text{for}\hspace{0.5ex} 0\leq r\leq R(t), t>0,$$
where $\hat{\sigma}:=\max\{1,\max{\sigma_{0}(r)}\}$. Note that $\mu(E)$ is bounded,
\begin{equation}\label{muEbd}\eta_{1}<\mu(E(r,t))<\eta_{2},\hspace{0.5ex}\text{for} \hspace{0.5ex}0\leq r\leq R(t), t>0\end{equation}
for some $0<\eta_{1}<\eta_{2}.$
\begin{thm}
	There exists a positive constant $\delta$, such that 
	$$R(t)\geq \delta\hspace{0.5em}\text{for all} \hspace{0.5em} t>0.$$
\end{thm}
\begin{proof}
	We will prove this theorem in two steps:
	
	\textbf{Step 1.} We claim that $\limsup_{t\rightarrow\infty}R(t)>0.$
	We shall argue by contradiction. If the conclusion is false, then
	\begin{equation}\label{RR}\lim_{t\rightarrow\infty} R(t)=0.\end{equation}
	As in the reference \cite{JMBFF}, we let 
	\begin{equation}\label{vv}v(r,t)=\bar{\sigma}\frac{R(t)}{r}\frac{\sinh (Mr)}{\sinh (MR(t))}\hspace{0.5em}\text{for}\hspace{0.5em} t\geq t_0,\end{equation}
	where $M^{2}=\lambda+2+N$ and $N$ is a positive number. As $r\rightarrow 0^{+}$, we have that
	$$\frac{r}{\sinh r}=1-\frac{1}{6}r^2+\frac{7}{360}r^{4}+O(r^{6}),$$
	$$\frac{\sinh r}{r}\frac{d}{dr}\bigg(\frac{r}{\sinh r}\bigg)=-\frac{1}{3}r+\frac{1}{45}r^{3}+O(r^{5}).$$
	Note $\mu(E), \sigma$ are bounded, we assume that 
	$$|\mu(E)(\sigma-\bar{\sigma})|<\eta, \hspace{0.5em} r<R(t), t>0.$$
	Take a small $\delta_{0}$ satisfying 
	$$0<\delta_{0}<\sqrt{\frac{6}{M^{2}\eta}},$$
	then for $R\in(0,\delta_{0})$
	$$-\frac{MR}{3}<\bigg(\frac{\sinh r}{r}\frac{d}{dr}\bigg(\frac{r}{\sinh r}\bigg)\bigg)\bigg|_{r=MR}<0,$$
	$$1>\frac{MR}{\sinh(MR)}>\max\bigg\{\frac{\eta M^{2}\delta^2_{0}}{6},\frac{1+\bar{\sigma}}{2}\bigg\}.$$
	From \eqref{RE},
	$$|R'(t)|<\frac{\eta R(t)}{3}\hspace{0.5em}\text{for all}\hspace{0.5em} t>0.$$
	Thus, 
	$$\begin{aligned}
	\left|\frac{dv}{dt}\right|&=\left|v(r,t)\frac{\sinh r}{r}\frac{d}{dr}\bigg(\frac{r}{\sinh r}\bigg)\right|_{r=MR}|MR'(t)|\\
	&\leq\frac{M^{2}}{3} R(t)|R'(t)|\\
	&\leq\frac{\eta M^2}{3}(R(t))^{2}.
	\end{aligned}$$
	
	Then, if $R(t)<\delta_{0}$, we have 
	$$\begin{aligned}
	v_{t}-\Delta v+\lambda v&=v_{t}-2v-Nv\\
	&<v_{t}-2v\\
	&\leq \frac{\eta M^{2}R^{2}}{3}-2\frac{\eta M^{2}\delta^2_{0}}{6}\\
	&<0.
	\end{aligned}$$
	From \eqref{RR}, there is a large time $T_{1}$ such that $R(t)<\delta_{0}, t\geq T_{1}$ and
	\begin{equation}v_{t}-\Delta v+\lambda v\leq 0, \hspace{0.5em}\text{if}\hspace{0.5em} r<R(t), t\geq T_{1}.
	\end{equation}
	Consider the function
	$$w=\sigma-v+z,$$
	where 
	$$z=e^{-\lambda(t-T_{1})}.$$
	
	It satisfies
	$$w_{t}-\Delta w+\lambda w\geq0, \hspace{0.5em}\text{if}\hspace{0.5em}r<R(t), t\geq T_{1},$$
	and it is positive 	on $r=R(t), t>T_{1}$ and on $\{t=T_{1}, r<R(T_{1})\}.$ By the maximum principle,
	$w>0 \hspace{0.5em}\text{if}\hspace{0.5em} t>t_{0},$ i.e.,
	$$\sigma(r,t)\geq v(r,t)-z(t).$$
	This inequality can be used to estimate $R'$ from below:
	$$\begin{aligned}
	R^2(t){R'}(t)&=\int_0^{R(t)}\mu(E)(\sigma(r,t)-\bar{\sigma})\rho^{2}d\rho\\
	&\geq\int_{0}^{R(t)}\mu(E)(v-\bar{\sigma})\rho^2d\rho-\int_{0}^{R(t)}\mu(E)z(t)\rho^2d\rho\\
	&\geq \frac{\eta_1}{6}(1-\bar{\sigma})R^{3}(t)-\frac{\eta_2}{3}e^{-\lambda(t-T_{1})}R^{3}(t)\\
	&>0.
	\end{aligned}$$	
	if $t$ is sufficiently large by \eqref{muEbd}. It follows that for some large time $T_{2}\hspace{0.5em} (T_2>T_1)$
	$${R'}(t)>0\hspace{0.5ex}\text{if}\hspace{0.5ex} t>T_2,$$
	i.e. $R(t)$ is monotone increasing. This contradicts to \eqref{RR}. Thus we finish the proof of step 1.
	
	\textbf{Step 2.} We claim that $\liminf_{t\longrightarrow \infty} R(t)\geq \delta_2,$ where $\delta_{1}=\theta\delta_{0}$ and $\delta_{2}=\theta^{2}\delta_{0}.$ To prove the above result, suppose for contradiction that $\liminf_{t\longrightarrow\infty}R(t)<\delta_{2}.$
	
	Then by the step 1, there exists a $t=t_{0}$ such that $R(t_{0})=\delta_{0}$. We shall prove that 
	\begin{equation}\label{Rall}R(t)\geq \delta_{2} \hspace{0.5ex} \text{for all}\hspace{0.5em}{t>t_{0}},\end{equation}
	and this establishes the theorem. Suppose that \eqref{Rall} is false, then there exists $t_{0}<t_{1}<t_{2}$ such that $R(t_{1})=\delta_{1}, R(t_{2})=\delta_{2}$ and 
	\begin{equation}\label{RRR}\begin{aligned}\delta_{2}<R(t)<\delta_{1} &\hspace{0.5em}\text{for}\hspace{0.5em} t_1<t<t_2,\\
	R'(t_{2})&\leq 0.\end{aligned}\end{equation}
	By \eqref{LRLR23R} and $R'/R\geq -\eta$, we have that
	$$\begin{aligned}
	t_{1}-t_{0}\geq\frac{1}{\eta}\log\frac{\delta_{0}}{\delta_{1}}=\frac{1}{\eta}\log\frac{1}{\theta}=\gamma,\\
	t_{2}-t_{1}\geq\frac{1}{\eta}\log\frac{\delta_{1}}{\delta_{2}}=\frac{1}{\eta}\log\frac{1}{\theta}=\gamma.\\
	\end{aligned}$$
	The domain 
	$$D_{1}=\{(r,t): r<\rho_{1}=\delta_{1}e^{\eta\gamma}, t_{1}-\gamma<t<t_{1}\}$$
	contains the sub-domain $D_{0}=\{(r,t): r<R(t), t_{1}-\gamma<t<t_{1}\}$. We introduce the solution $W$ to 
	\begin{equation}\begin{cases}
	W_{t}=\Delta W-\lambda W, \hspace{0.5em}\text{in}\hspace{0.5em} D_{1},\\
	W=1, r=\rho_{1}, t\in(t_{1}-\gamma, t_{1}),\\
	W=0, r<\rho_{1}, t=t_{1}-\gamma.
	\end{cases}\end{equation}
	 By comparison with Green's function for a rectangular domain constructed by a series of reflections, we obtain
	$$W(r, t_{1})\geq \epsilon_{0}>0, r<\rho_{1},$$
	then by maximum principle, we have $\sigma\geq w$ in $D_1$, thus
	$$\sigma(r,t_{1})\geq\epsilon_{0},\hspace{0.5em}r<\rho_{1}.$$
	Next we introduce the domain 
	\begin{equation}\label{DD}D_{2}=\{(r,t): r<R(t), t_{1}<t<t_{2}\} \end{equation}
	and a comparison function in $D_{2}:$
	$$v(r,t)=\hat{\sigma}(t)\frac{R(t)}{r}\frac{\sinh(Mr)}{\sinh(MR(t))},$$
	where
	\begin{equation}\label{hats}\hat{\sigma}(t)=\epsilon_{0}e^{N(t-t_1)}, t_{1}<t<t_{2},\end{equation}
	$$N=\frac{1}{t_2-t_1}\log\frac{1}{\epsilon_0},$$
	so that $\hat{\sigma}(t_{1})=\epsilon_{0}, \hat{\sigma}(t_2)=1.$ As in Step 1 we compute 
	$$\begin{aligned}
	v_{t}-\Delta v+\lambda v&=\frac{\hat{\sigma}'(t)}{\hat{\sigma}(t)}v+\hat{\sigma}(t)(\partial_{t}-\Delta+\lambda)\frac{v}{\hat{\sigma}(t)}\\
	&=Nv+\hat{\sigma}\frac{\partial}{\partial t}\left(\frac{v}{\hat{\sigma}(t)}\right)-2v-Nv\\
	&=\hat{\sigma}\frac{\partial}{\partial t}(\frac{v}{\hat{\sigma}(t)})-2v\\
	&\leq \hat{\sigma}(t)\left[\frac{\eta M^{2}}{3}R^{2}-2(1+o(R^{2}))\right]\\
	&<0.
	\end{aligned}$$
	Thus $v$ is a subsolution.  In view of \eqref{DD} and \eqref{hats}, we see that $\sigma\geq v$ on both $r=R(t)$ and $t=t_{1}$. Here the maximum principle implies $\sigma>v$ in $D_2$ and 
	$$\sigma(r,t_2)\geq \frac{R(t_2)}{r}\frac{\sinh(Mr)}{\sinh(MR(t_2))}$$
	and 
	$$\sigma(r,t_2)\geq \frac{M\delta_2}{\sinh(M\delta_2)}>\frac{1+\bar{\sigma}}{2}.$$
	Using this in \eqref{RE} we deduce that $R'(t_2)>0,$ a contradiction to \eqref{RRR}.
\end{proof}

\section{The existence of radially symmetric stationary solution}

In this section, we derive the existence and uniqueness of the radially symmetric stationary solution. The major challenge for establishing existence and uniqueness stems from the singularity of our integro-differential equation. These types of equations are not covered by the standard theory. Another challenge is to establish continuity of the velocity field near both ends $r=0$ and $r=R_{s}$. In addition, for such a highly nonlinear system, the uniqueness is by no means trivial. As a matter of fact, uniqueness may not be valid for certain system (e.g., stationary problem for the protocell \cite{FriedmanHu-Siam99}). We have explored the special structure of our problem which enables us to overcome the difficulties and established uniqueness. 

In order to establish the existence, another important work in our paper is to construct the well-posed of the general singular equations. Since the proof is lengthy and complex, we put it in the appendix.

We consider the general singular integro-equation
\begin{equation}\label{eq:A1}\begin{cases}
\frac{dx}{dt}=\frac{f(x,t)}{\int_{0}^{t} g(x(s),s)ds},\\
x(0)=x_{0},
\end{cases}\end{equation}
where $f$, $g$ are two functions defined in a domain $G\subset\mathbb{R}\times\mathbb{R}$. We assume that $f$, $g$ satisfy

\begin{enumerate}[label=(F\arabic*), start=1]
	\item\label{item11} $(x_{0},0)\in G$ and $f(x_{0},0)=0$;
	\item\label{item12} $f$ and its derivatives $f_{x}, f_{t}$ are continuous in $G$;
	\item\label{item13} $g\in C(G)$ and $g$ is local Lipschitz continuous with respect to $x$;
	\item\label{item14} $g(x_{0},0)\neq0$ and $\theta=f_{x}(x_{0},0)/g(x_{0},0)\in(-\infty,1)$.
\end{enumerate}

\begin{thm}\label{thm:S03A}
	Assume $f$ and $g$ satisfy conditions \ref{item11} though \ref{item14}. Then \eqref{eq:A1} admits a unique $C^{1}$ solution in $[-\mathcal{T},\mathcal{T}]$ for a small $\mathcal{T}.$
	Moreover,
	$$x'(0)=\frac{f_{t}(x_{0},0)}{g(x_{0},0)-f_{x}(x_{0},0)}.$$
\end{thm}

\begin{rmk}
	The above theorem guarantees uniqueness in the class $C^{1}[0,\mathcal{T}]$. In the case $\theta<0$, the continuous solution $x\in C^{0}[0,\mathcal{T}]\cap C^{1}(0,\mathcal{T}]$ of \eqref{eq:A1} exists and is unique, and it is also in $C^{1}$ class. However, it is crucial to notice that in the case $\theta\in(0,1)$, the continuous solution of \eqref{eq:A1} may not be unique, but the $C^{1}$ solution is unique.
\end{rmk}

\begin{rmk}
	In fact, \autoref{thm:S03A} tells us the solution exists in a small interval. If one denotes by $y=y(t)$ the denominator of \eqref{eq:A1}, then $x=x(t),y=y(t)$ satisfy a  integro-differential equation, and the solution can be expanded to a maximal existence interval.
\end{rmk}

We next consider an equation of the type \eqref{eq:A1} involving a parameter $\mu$:
\begin{equation}\label{eq:A2}\begin{cases}
\frac{dx}{dt}=\frac{f(x,t,\mu)}{\int_{0}^{t} g(x(s),s,\mu)ds},\\
x(0)=\varphi(\mu),
\end{cases}\end{equation}
where $f$ and $g$ are defined in an open set $G\times\mathcal{U}$ of $(\R\times\R)\times\R^{m}$ with $(x_{0},0)\in G$, $\mu_{0}\in\mathcal{U}$.
We assume that $f$ and $g$ satisfy:

\begin{enumerate}[label=(F\arabic*),resume]
	\item\label{item21} $\varphi(\mu)$ is a continuous function in a neighborhood of $\mu_{0}$,
	and $\varphi(\mu_{0})=x_{0}$, $f(0,\varphi(\mu),\mu)\equiv0$.
	\item\label{item22} $f,g\in C(G\times\mathcal{U})$ are local Lipschitz continuous with respect to $x$ in $G\times\mathcal{U}$.
	\item\label{item23} The derivatives $f_{x}, f_{t}$ exist and are continuous in a neighborhood of $(x_{0},0,\mu_{0})$.
	\item\label{item24} $g\neq0$ and $f_{x}/g<1$ at $(x,t,\mu)=(x_{0},0,\mu_{0})$.
\end{enumerate}

\begin{thm}\label{thm:S03B}
	Let $f, g$ be the functions satisfying conditions \ref{item21}-\ref{item24}.
	Assume that the unique solution $x=x(t,\mu_{0})$ of \eqref{eq:A2} exists in a bounded closed interval $[a,b]$ with $a<0<b$.
	Then there exists a $\delta>0$ such that
	
	(1) For every $\mu\in B_{\mu_{0},\delta}$, the solution $x=x(t,\mu)$ of \eqref{eq:A2} exists in $[a,b]$.
	
	(2) The function $(t,\mu)\mapsto x(t,\mu)$ is continuous in $[a,b]\times B_{\mu_{0},\delta}$.
\end{thm}
 Next we consider steady state equations in the radially symmetric case are:

\begin{subequations}\label{sss02}\begin{align}
	\label{LLL02Sa} &\frac{1}{r^{2}}\frac{\partial}{\partial r} \left(r^{2}\frac{\partial\sigma}{\partial r}\right)-\lambda\sigma=0,\hspace{1em} 0<r<R,\\
	\label{sss02ma} &D_m\frac{1}{r^{2}}\frac{\partial}{\partial r}\left(r^{2}\frac{\partial m}{\partial r}\right)+\alpha-\beta m=0,\hspace{1em} 0<r<R,\\
	\label{sss02Ea} &u E'=Q(\sigma, m, E):=-\gamma m E + \phi(E)-E\mu(E)(\sigma-\bar{\sigma}),\hspace{1em} 0<r<R,\\
	\label{sss02ua} & u'+\frac{2}{r} u=\mu(E)(\sigma-\bar{\sigma}),\hspace{1em} 0<r<R,\\
	\label{sss02Sb} &\sigma'(0)=0,\hspace{0.5em}\sigma(R)=1,\\
	\label{sss02mb} &m'(0)=0,\hspace{0.5em}m'(R)=0,\\
	\label{sss02ub}&u(0)=0, u(R)=0.
	\end{align}\end{subequations}
It is clear from \eqref{sss02ma} and \eqref{sss02mb} that $m=\frac{\alpha}{\beta}.$
\subsection{The stationary solution for $\sigma$}
\par\vspace*{0.5ex}\par\noindent

We impose the following structural conditions: for some $N>0$,
\begin{equation}\label{sss03Q}\begin{aligned}
\frac{\partial Q(\sigma, m, E)}{\partial E}+\mu(E)\bar{\sigma}<0, &\hspace{1em} 0<\sigma\leq 1,\hspace{0.5em} m=\frac{\alpha}{\beta},\hspace{0.5em}0<E<\infty,\\
Q(\sigma,m,N)<0, &\hspace{1em} 0<\sigma\leq 1,\hspace{0.5em} m=\frac{\alpha}{\beta}.
\end{aligned}\end{equation}
Condition \eqref{sss03Q} also implies that $Q(\sigma, m, E)$ is decreasing in $E>0$, that is,
$$\frac{\partial Q(\sigma, m, E)}{\partial E}<0.$$
\begin{lem}\label{lem:201}
  Assume that there is a solution $(\sigma,m,E,u,R)$ to \eqref{sss02}. Then
  \begin{equation}\label{sss04Sm}
  \sigma(r)=\frac{R}{\sinh(\sqrt{\lambda}R)}\frac{\sinh(\sqrt{\lambda}r)}{r},\hspace{2ex} m=\frac{\alpha}{\beta}.
  \end{equation}
  $u$ satisfies $u(0)=0$, $u(R)=0$, $u(r)<0$ for $r\in(0,R)$ and 
  \begin{equation}\label{sss04u}
  u(r)=\frac{1}{r^{2}}\int_{0}^{r}\mu(E(\rho))(\sigma(\rho)-\bar{\sigma})\rho^{2}d\rho.
  \end{equation}
  $E\in C^{1}[0,R]$ satisfies the following singular equation
  \begin{equation}\label{sss04E}
  E'(r)=\frac{Q(\sigma(r),m, E(r))}{\frac{1}{r^{2}}\int_{0}^{r}\mu(E(\rho)) (\sigma(\rho)-\bar{\sigma})\rho^{2}d\rho},
  \end{equation}
  and $E(0)$, $E(R)$ satisfy
  \begin{equation}\label{sss04Eb}Q(\sigma(r),m,E(r))=0\mbox{ at }r=0,R.\end{equation}
  Finally, $R>0$ satisfies
  \begin{equation}\label{sss04R}
  \sigma(0)=\frac{\sqrt{\lambda}R}{\sinh(\sqrt{\lambda}R)}<\bar{\sigma}.
  \end{equation}
\end{lem}
\begin{proof}
We immediately obtain $\sigma$, given explicity by \eqref{sss04Sm}. Note that from \eqref{sss02Ea} and \eqref{sss02ub}, we get the $C^{1}$ solution $E$ satisfies
$$Q(\sigma(r),m,E(r))=0\mbox{ at }r=0,R.$$
From \eqref{sss02ua} with \eqref{sss02ub} and the fact that $\bar{\sigma}\in(0,1)$, we find that the solution $u$ is represented as
$$u(r)=\frac{1}{r^{2}}\int_{0}^{r}\mu(E(\rho))(\sigma(\rho)-\bar{\sigma})\rho^{2}d\rho
\mbox{ or }
u(r)=\frac{1}{r^{2}}\int_{R}^{r}\mu(E(\rho))(\sigma(\rho)-\bar{\sigma})\rho^{2}d\rho.$$
\eqref{sss04E} then is a re-statement of \eqref{sss02Ea}.

Note that $u(0)=u(R)=0$, $\mu(E)$ is positive, and $\sigma(r)-\tilde{\sigma}$ is strictly increasing, we derive that $\sigma-\tilde{\sigma}$ admits exactly one interior root $r_{0}\in(0,R)$ and $R$ satisfies \eqref{sss04R}. Therefore, $u$ is negative in $(0,R)$.
\end{proof}

Since $\sigma(r)$ also depends on $R$, we write it as $\sigma(r,R)$.
From the assumption \eqref{sss03Q} and the fact $Q(\sigma, m, 0)=\phi(0)>0$, there exist a unique $h=h(r,R)>0$ such that
\begin{equation}\label{sss03Qa}Q(\sigma(r,R), m, h)=0. \end{equation}
Moreover, the implicit function theorem implies that $h(r,R)$ is a smooth function in two variable $(r,R)$ and that
\begin{equation}\label{sss03Qc}
\frac{\partial h}{\partial r}
=-\frac{Q_{\sigma}(\sigma(r,R),m,h(r,R))}{Q_{E}(\sigma(r,R),m,h(r,R))}
\frac{\partial\sigma(r,R)}{\partial r}.
\end{equation}
From \eqref{sss03Q} and $Q_{\sigma}=-\mu(E)E<0$, we conclude that $h$ is a decreasing function in $r$,
$$h_{r}(r,R)<0, r>0.$$
We also have $h$ is uniformly bounded by positive constants from below and above,
\begin{equation}\label{sss03Qd} N_{1}<h(r,R)<N, r\in[0,R],\end{equation}
where $N_{1}=h(R,R)>0$ is a constant independent of $R$ (since $\sigma(R,R)\equiv1$). Therefore,
\begin{equation}\label{sss03Qb}Q(\sigma(r,R), m, \xi)=\begin{cases}
<0 &\mbox{ for }\xi>h(r,R),\\
=0 &\mbox{ for }\xi=h(r,R),\\
>0 &\mbox{ for }\xi<h(r,R).
\end{cases}\end{equation}

In order to find a solution $(\sigma,m,E,u,R)$ of \eqref{sss02}, it is equivalent to obtain a solution $(E, u, R)$ of the following equations
\begin{subequations}\label{LLL12}\begin{align}
\label{LLL12E}&uE'=Q(\sigma(r,R),m,E),\\
\label{LLL12U}&u'+\frac{2}{r}u=\mu(E)(\sigma(r,R)-\bar{\sigma}),\\
\label{LLL12U1}&u(R)=0, E(R)=h(R,R),\\
\label{LLL12U0}&u(0)=0, E(0)=h(0,R),
\end{align}\end{subequations}
where
\begin{equation}\label{LLL13psi}
\sigma(r,R)=\frac{R}{\sinh(\sqrt{\lambda}R)}\frac{\sinh(\sqrt{\lambda}r)}{r}.
\end{equation}

Note that problem \eqref{LLL12} is an ordinary differential system, and there are singularities at $r=0, R$.
Since problem \eqref{LLL12} is a boundary value problem, we will solve \eqref{LLL12} by the shooting method.
For any fixed $R>0$, we solve the initial problem $(E,u)$ of \eqref{LLL12E}-\eqref{LLL12U1} near $r=R$,
then find a suitable $R$ such that $(E,u)$ exists in $[0,R]$ and satisfies condition \eqref{LLL12U0} at the other end.

\subsection{The existence for  stationary solution $E$}
\par\vspace{2ex}\par\noindent

\textbf{Solutions near $r=R$, i.e., $0<R-r\ll1.$ } If we integrate \eqref{LLL12U} with $u(R)=0$, we get
$$u(r)=\frac{1}{r^{2}}\int_{R}^{r}\mu(E)(\sigma-\bar{\sigma})\rho^{2}d\rho.$$
Therefore, $E$ satisfies the following initial value problem:
\begin{equation}\label{LLL16}\begin{cases}
\frac{dE(r)}{dr}=
\frac{Q(\sigma(r),m,E(r))}
{\frac{1}{r^{2}}\int_{R}^{r}\mu(E(\rho))(\sigma(\rho)-\bar{\sigma})\rho^{2}d\rho},\\
E(R)=h(R).
\end{cases}\end{equation}

\begin{lem}\label{lem:204}
  The equation \eqref{LLL16} admits a unique solution. Moreover,
  \begin{equation}\label{LLL05B}
  \frac{dE}{dr}(R)
  =\frac{E(R)\mu(E(R))\sigma_{r}(R)}
  {\frac{\partial Q}{\partial E}(1,m, E(R))-\mu(E(R))(1-\bar{\sigma})}.
  \end{equation}
\end{lem}
\begin{proof}
Note that $E(R)=h(R)>0$ and $Q_{E}<0$, we derive
$$\tilde{\theta}\triangleq\frac{\gamma m-\phi_{E}(R)-\frac{\partial(E\mu(E))}{\partial E}(\bar{\sigma}-\sigma)}{\mu({E(R)})(\bar{\sigma}-1)} = -\frac{\frac{\partial Q}{\partial E}|_{r=R}}{\mu(E(R))(\bar{\sigma}-1)}<0.$$
Therefore, the conditions in \autoref{thm:S03A} are satisfied. We can, after a change of variables $r\rightarrow R-r,$ apply \autoref{thm:S03A} to show the existence and uniqueness of the solution $E\in C^{1}[R-\delta, R]$ to \eqref{sss02}, for a small $\delta>0.$
\end{proof}

\autoref{lem:204} tells us the solution $E$ of \eqref{LLL16} exists in the interval $[R-\delta,R]$, it also exists in a right neighborhood of $r=R$.
The equations \eqref{LLL16} or \eqref{LLL12E}-\eqref{LLL12U} do not have a singularity at $r=R-\delta$.
The standard ODE theory implies that the solution pair $(E,u)$ can be extended to a maximal interval $(\tau,R]$, so long as the denominator $u$ is negative and $E$ is positive on $(\tau,R)$. Noting that there is a singularity terms $2u/r$ in equation \eqref{LLL12U}, we see that $\tau\geq0$.
It is clear that
\begin{equation}\label{LLL19} E(r)>0,\hspace{2ex} u(r)<0,\hspace{2ex}r\in(\tau,R).\end{equation}
Since $R$ is also unknown, we usually denote the solution by $E(r,R)$, $u(r,R)$ and denote by $\tau=\tau(R)$ for the maximum existence interval $(\tau(R),R]$.

\begin{rmk}\label{rmk:205}
  The initial solution $(E,u)$ of \eqref{LLL12E}, \eqref{LLL12U} and \eqref{LLL12U1} is also depending continuously on $R$; see \autoref{thm:S03B}. More precisely, if for any fixed $\bar{R}>0$, the solution $E(r,\bar{R})$, $u(r,\bar{R})$ of \eqref{LLL12E}, \eqref{LLL12U} and \eqref{LLL12U1} are well-solved in a compact interval $[\hat{r},\check{r}]$ with $\bar{R}\in(\hat{r},\check{r})$, then there exists $\delta>0$ such that

  (i) $E(r,R)$, $u(r,R)$ of \eqref{LLL12E}, \eqref{LLL12U} and \eqref{LLL12U1} are well-solved in the interval $[\hat{r},\check{r}]$ for every $R\in[\bar{R}-\delta,\bar{R}+\delta]$.

  (ii) $E(r,R)$, $u(r,R)$ are continuous functions in the variable $(r,R)\in [\hat{r},\check{r}]\times[\bar{R}-\delta,\bar{R}+\delta]$.
\end{rmk}
\par\vspace{2ex}\par\noindent 
\textbf{Solution near $0<r\ll1.$} We consider
$$\frac{dE}{dr}=\frac{Q(\sigma(r),m,E)}{\frac{1}{r^{2}}\int_{0}^{r}\mu(E)(\sigma-\bar{\sigma})\rho^{2}d\rho}.$$
The singularity at $r=0$ is not exactly of the format in \autoref{thm:S03A}, but the quantity $\theta$ defined in \ref{item14} proceeding \autoref{thm:S03A} is very close to
$$\begin{aligned}
\tilde{\theta}=
&3\frac{\gamma m-\phi_E(0)-\frac{\partial(E\mu(E))}{\partial E}(\bar{\sigma}-\sigma)}{\mu(E_{0})(\bar{\sigma}-\sigma(0, R))}
= -3\frac{\frac{\partial Q}{\partial E}|_{r=0}}{\mu(E_{0})(\bar{\sigma}-\sigma(0))}>0.
\end{aligned}$$
It is a very heavy assumption to require $\tilde{\theta}<1,$ and in some special cases of interest, this is not possible. That is why we solve the ODE system starting from $r=R.$

Therefore, we use a different approach. Instead of working from $r=0$, we shall solve the ODE starting from $r= R.$

\begin{lem}[The bounds for $E$]\label{lem:206}
There holds
$$E(R)<E(r)<h(r,R),\hspace{1em} r\in(\tau,R).$$
Moreover, $E\in C^{1}(\tau, R]$, $E'(r)<0$ for $r\in(\tau, R).$ Furthermore, if we define $E(\tau)=\lim_{r\rightarrow\tau+0} E(r)$, then $E\in C[\tau, R].$
\end{lem}
\begin{proof}
Since the derivatives of $E$ and $h$ at $r=R$ are given in \eqref{sss03Qc} and \eqref{LLL05B}, we deduce that
$$\frac{d(E-h)}{dr}\Big|_{r=R}=\left[\frac{E\mu(E)\sigma_{r}}{Q_E(1, m, E)-\mu(E)(1-\bar{\sigma})}-\frac{E\mu(E)\sigma_{r}}{Q_E(1, m, E)}\right]\bigg|_{r=R}>0.$$
Since $E(R)=h(R)$, we get $E(r)<h(r)$ for $0<R-r\ll1$.

We claim that $E(r)<h(r)$ for all $r\in(\tau, R)$. If there exists $r_{1}\in(\tau,R)$ such that
$$E(r_{1})=h(r_{1}, R),~~E(r)<h(r,R),r\in(r_{1},R),$$
then $E_{r}(r_{1})\leq h_{r}(r_{1})$. However, $h'<0$ for all $r>0$ and $E'(r_{1})=0$ by \eqref{LLL16}, this is a contradiction. Therefore, $E(r)-h(r)<0$ and also $E'(r)<0$ for $r\in(\tau,R)$, again by \eqref{LLL16}.

From \eqref{sss03Q} or \eqref{sss03Qd}, we see that $E$ is bounded and monotone decreasing in $(\tau,R]$. Thus, the limit $\lim_{r\rightarrow\tau+0} E(r)$ exists and is positive, and is denoted by $E(\tau)$. Therefore, $E\in C[\tau, R]\cap C^{1}(\tau,R].$
\end{proof}

Now we set $I(r)=r^{2}u(r)$ or
$$I(r)\triangleq\int_{R}^{r}\mu(E(\rho))(\sigma(\rho)-\bar{\sigma})\rho^{2}d\rho.$$
As a consequence of \autoref{lem:206}, we find $I(r)$ and its derivatives $I'(r)$ are continuous up to $r=\tau$, say, $I\in C^{1}[\tau,R]$. Thus $\tau=\tau(R)$ satisfies:
$$\tau=\inf\{\tau'\in(0,R):I(r)<0\texttt{ for }r\in(\tau',R)\}.$$

We have the following more precise results about values at $r=\tau$.
\begin{lem}\label{lem:207}
We have that

(i) If $\tau\in(0,R)$, then $I(\tau)=0$, $u(\tau)=0$.
Moreover, $E(\tau)=h(\tau)>0$, $E\in C^{1}[\tau,R], E'(\tau)<0,$ and
\begin{equation}
E'(\tau)=\frac{E(\tau)\mu(E(\tau))\sigma_{r}(\tau)}{Q_E(\sigma(\tau), m, E(\tau))-\mu(E(\tau))(\sigma(\tau)-\bar{\sigma})}.
\end{equation}

(ii) If $\tau=0$ and $I(0)=0$, then $u(0)=0$, $E(0)=h(0)$, $E\in C^{1}[0,R]$ and $E'(0)=0$.
\end{lem}

\begin{proof}
\textbf{The case $\tau>0$.} Noting that $(\tau,R]$ is the maximal existence interval, we get $u(\tau)=0$ and $I(\tau)=0$. From $I(\tau)=I(R)=0$ and
$$I'(r)=\mu(E(r))(\sigma(r)-\bar{\sigma})\rho^{2},\mbox{ and }\sigma_{r}>0,$$
we see that $\sigma(r)-\bar{\sigma}=0$ has a root in $[\tau,R]$. Moreover, the root is unique and in $(\tau,R)$, $\sigma(\tau)-\bar{\sigma}<0$.
Since $E\in C[\tau,R]$ and
$$u(r) = \frac{1}{r^{2}}\int_{R}^{r}\mu({E})(\sigma(\rho)-\bar{\sigma})\rho^{2}d\rho = \frac{1}{r^{2}}\int_{\tau}^{r}\mu({E})(\sigma(\rho)-\bar{\sigma})\rho^{2}d\rho,$$
we find $u\in C^{1}[\tau,R]$ and $u(\tau)=0$, $u_{r}(\tau)=\mu(E(\tau))(\sigma(\tau)-\bar{\sigma})<0$.

If $E(\tau)\neq h(\tau)$, then $C_{1}=Q(\sigma(\tau, R), m, E(\tau, R))\neq0$. Setting $C_{2}=u'(\tau)<0$, we deduce from equation \eqref{LLL16} that
$$E'=\frac{1}{r-\tau}\frac{C_{1}+o(1)}{C_{2}+o(1)}\mbox{ near } r=\tau.$$
Integrating the above equality, we see that $E$ is unbounded near $r=\tau$. This yields a contradiction. Hence, $E(\tau)=h(\tau, R)$.

From \eqref{sss03Q}, we get
$$\begin{aligned}
  \theta=&\frac{Q_{E}(\sigma(\tau),m,E(\tau))}{u_{r}(\tau)}
  =\frac{Q_{E}(\sigma(\tau),m,E(\tau))}{\mu(E(\tau))(\sigma(\tau)-\bar{\sigma})}\\
  =&\frac{Q_{E}(\sigma(\tau),m,E(\tau))-\mu(E(\tau))(\sigma(\tau)-\bar{\sigma})}
  {\mu(E(\tau))(\sigma(\tau)-\bar{\sigma})}+1\\
  >&1.
\end{aligned}$$
Now we apply \autoref{lem:FH2008} for equation \eqref{LLL12E} to obtain that $E\in C^{1}[\tau,R]$ and
$$E_{r}(\tau)=\frac{Q_{\sigma}\sigma_{r}}{u_{r}-Q_{E}}|_{r=\tau}
=\frac{E(\tau)\mu(E(\tau))\sigma_{r}(\tau)}{Q_E(\sigma(\tau),m, E(\tau)) - \mu(E)(\sigma(\tau)-\bar{\sigma})}<0.$$

\textbf{The case $\tau=0$ and $I(0)=0$.}
Noting that
$$u(r) = \frac{1}{r^{2}}\int_{R}^{r}\mu(E)(\sigma(\rho)-\bar{\sigma})\rho^{2}d\rho = \frac{1}{r^{2}}\int_{0}^{r}\mu(E)(\sigma(\rho)-\bar{\sigma})\rho^{2}d\rho,$$
we get that $u$ is differentiable at $r=0$, and $u'(0)=\mu(E(0))(\sigma(0)-\bar{\sigma})/3<0$.

If $C_{1}\neq0$ where $C_{1}=Q(\sigma(0), m, E(0))$, then from equation \eqref{LLL16},
$$E'=\frac{1}{r}\frac{C_{1}+o(1)}{C_{2}+o(1)}\mbox{ near }r=0,$$
where as before $C_{2}=u'(0)=\mu(E(0))(\sigma(0)-\bar{\sigma})/3<0$.
This contradicts the boundedness of $E$ and therefore, $C_{1}=Q(\sigma(0), m, E(0))=0$,  $E(0)=h(0)$.

From \eqref{sss03Q}, we get
$$\begin{aligned}
  \theta=&\frac{Q_{E}((\sigma(0),m,E(0)))}{u_{r}(0)}
  =\frac{Q_{E}((\sigma(0),m,E(0)))}{\mu(E(0))(\sigma(0)-\bar{\sigma})/3}\\
  =&\frac{3Q_{E}((\sigma(0),m,E(0)))-\mu(E(0))(\sigma(0)-\bar{\sigma})}
  {\mu(E(0))(\sigma(0)-\bar{\sigma})}+1\\
  >&1.
\end{aligned}$$
Applying \autoref{lem:FH2008} for equation \eqref{LLL12E}, we obtain that $E\in C^{1}[\tau,R]$ and
$$E'(0)=\frac{3E(0)\mu(E(0))\sigma_{r}(0)}
{3Q_E(\sigma(0),m)-\mu(E(0))(\sigma(0)-\bar{\sigma})}=0.$$
This completes the proof.
\end{proof}
\begin{rmk}
  The assumption \eqref{sss03Q} on $Q$, used in \autoref{lem:207}, only guarantee that the solution $E$ obtained in \autoref{lem:204} is differentiable at the minimal existence value $r=\tau$. So the radially symmetric stationary solution obtained later in \autoref{thm:211} satisfies $E\in C^{1}[0,1]$.
\end{rmk}

\subsection{Continuous dependence with respect to the parameter $R$}

\par\vspace{2ex}\par\noindent 

In this subsection, we give some estimates of $\tau=\tau(R)$ and the continuous dependence with respect to $R$.

\begin{lem}\label{lem:208a}
If $0<R\ll 1$, then $\tau(R)=0$ and
$$I(r,R)<0,\hspace{1em} 0\leq r<R.$$
\end{lem}
\begin{proof}
Noting that $0<\bar{\sigma}<1$, and
$$\lim_{R\rightarrow0}\sigma(0,R)-\bar{\sigma}
=\lim_{R\rightarrow0}\frac{R}{\sinh(\sqrt{\lambda} R)}-\bar{\sigma}
=1-\bar{\sigma}>0,$$
we see that for $0<R\ll1$,
$$\sigma(r,R)-\bar{\sigma}\geq\sigma(0,R)-\bar{\sigma}>0, r\in[0,R].$$
Observing that
$$I(R,R)=0, I'(r,R)=\mu(E)(\sigma(r)-\bar{\sigma})r^{2}>0,r\in[\tau(R),R),$$
we deduce that for $0<R\ll1$,
$$I(r,R)<0,r\in[\tau(R),R).$$
The proof immediately follows from \autoref{lem:207} (i).
\end{proof}

\begin{lem}\label{lem:208b}
If $R\gg 1,$ then $\tau(R)>0$ and
$$I(r,R)<0,\hspace{0.5em}\tau(R)<r<R,\hspace{0.5em}I(\tau(R),R)=0.$$
\end{lem}
\begin{proof}
We take $K$ such that
$e^{-\sqrt{\lambda}K}=\bar{\sigma}/3.$
Then
$$\lim_{R\rightarrow\infty}\sigma(R-K,R)
=\lim_{R\rightarrow\infty}\frac{R}{\sinh(\sqrt{\lambda}R)}\frac{\sinh(\sqrt{\lambda}(R-K))}{R-K}
=e^{-\sqrt{\lambda}K}=\frac{1}{3}\bar{\sigma}.$$
Thus, for $R\gg1,$
$$\sigma(r,R)-\bar{\sigma}<\sigma(R-K,R)-\bar{\sigma}<\frac{1}{2}\bar{\sigma},
\hspace{2ex}0\leq\rho\leq R-K.$$
From \eqref{sss03Qd} and \autoref{lem:206}, we conclude $E(r,R)$ is uniformly bounded,
$$N_{1}<E(r,R)<N, r\in(\tau(R),R),$$
where $N_{1}=E(R,R)=h(R,R)$ is independent of $R$. Thus, there exist two positive constants $K_{1}<K_{2}$ (they are independent of $R$) such that
\begin{equation}\label{LLL21} K_{1}<\mu(E(r,R))<K_{2}, r\in(\tau(R),R). \end{equation}
Therefore, if $\tau(R)=0$ for some sufficiently large $R$, then
$$\begin{aligned}
0&\leq\int_{0}^{R}\mu(E)({\sigma-\bar{\sigma}})\rho^{2}d\rho\\
&=\int_{0}^{R-K}\mu(E)(\sigma-\bar{\sigma})\rho^{2}d\rho + \int_{R-K}^{R}\mu(E)({\sigma-\bar{\sigma}})\rho^{2}d\rho\\
&<(-\frac{1}{2}\bar{\sigma})K_{1}\frac{(R-K)^3}{3} + 2K_{2}\frac{(R)^3-(R-K)^3}{3}\\
&<0,
\end{aligned}$$
this is a contradiction. Thus, $\tau(R)>0$ for all sufficiently large $R$.
\end{proof}
\begin{lem}\label{lem:209a}
The function $R\in(0,\infty)\mapsto\tau(R)$ is upper semi-continuous, that is
$$\tau(\bar{R})\geq\limsup_{R\rightarrow\bar{R}}\tau(R)$$
for every $\bar{R}>0$.
\end{lem}
\begin{proof}
Let $\bar{R}>0$ be fixed and $T\in(\tau(\bar{R}),\bar{R})$ be any fixed constant. Then from \autoref{rmk:205}, we see that
the solution $E(r,R),u(r,R)$ of \eqref{LLL12E}, \eqref{LLL12U} and \eqref{LLL12U1} exist at least for $r\in [T,R)$ and for $R\in[\bar{R}-\delta,\bar{R}+\delta]$ for some $\delta>0$. Therefore, $\tau(R)$ satisfies
$$\tau(R)<T\mbox{ for every }R\in[\bar{R}-\delta,\bar{R}+\delta],$$
and then
$$\limsup_{R\rightarrow\bar{R}}\tau(R)\leq T.$$
Since $T\in(\tau(\bar{R}),\bar{R})$ is arbitrary, we finish the proof.
\end{proof}

\begin{lem}\label{lem:209b}
The function $\tau(R)$ is lower semi-continuous, that is, for any $\bar{R}>0$, we have
\begin{equation}\label{LLL22}
\tau(\bar{R})\leq \liminf_{R\rightarrow\bar{R}}\tau(R).
\end{equation}
Thus, the function $R\in(0,\infty)\mapsto\tau(R)$ is continuous.
\end{lem}
\begin{proof}
Noting that $\tau(R)\geq0$ for $R>0$, we know \eqref{LLL22} holds when $\tau(\bar{R})=0$. Thus, we always assume that $\bar{\tau}=\tau(\bar{R})>0$.
Since $I(\bar{\tau},\bar{R})=0$ we find $\psi(\bar{\tau},\bar{R})<0,$ where $\psi(r,R)=\sigma(r,R)-\bar{\sigma}$. By continuity there exists $\delta_{1}>0$ such that
\begin{equation}\label{LLL23a}
2\psi(\bar{\tau},\bar{R})<\psi(r,R)<\frac{1}{2}\psi(\bar{\tau},\bar{R})<0
\end{equation}
for $r\in[\bar{\tau}-\delta_{1},\bar{\tau}+\delta_{1}]$ and $R\in[\bar{R}-\delta_{1},\bar{R}+\delta_{1}]$.

Suppose that \eqref{LLL22} is false, then there exists $\delta_{2}>0$ and a sequence of $\{R_{k}\}_{k\geq1}$ with $\lim_{k\rightarrow\infty}R_{k}=\bar{R}$ such that
\begin{equation}\label{LLL23b}\tau(R_{k})<\bar{\tau}-\delta_{2}, k\geq1.\end{equation}
Without loss of generality, we may assume that $0<\delta_{2}<\delta_{1}<\bar{\tau}$ and $R_{k}\in[\bar{R}-\delta_{1},\bar{R}+\delta_{1}]$, $k\geq1$.
In order to yield a contradiction, we will show that $I(\bar{\tau}-\delta_{2},R_{k})>0$ for large $k$. By noting $I(\bar{\tau},\bar{R})=0$, we split $I(\bar{\tau}-\delta_{2},R_{k})$ into three parts
$$\begin{aligned}
I(\bar{\tau}-\delta_{2},R_{k})=&[ I(\bar{\tau}-\delta_{2},R_{k})-I(\bar{\tau}+\delta,R_{k})]\\
& + [I(\bar{\tau}+\delta,R_{k})-I(\bar{\tau}+\delta,\bar{R})]
+ [I(\bar{\tau}+\delta,\bar{R})-I(\bar{\tau},\bar{R})].
\end{aligned}$$
From \eqref{sss02ua}, \eqref{LLL21} and \eqref{LLL23a}, we have
\begin{equation}\label{LLL24a}\begin{aligned}
I(\bar{\tau}-\delta_{2},R_{k})-I(\bar{\tau}+\delta,R_{k})
=&-\int_{\bar{\tau}-\delta_{2}}^{\bar{\tau}+\delta}\mu(E(r,R_{k}))\psi(r,R_{k})r^{2}dr\\
\geq&-\frac{1}{2}\psi(\bar{\tau},\bar{R})(\delta+\delta_{2})(\bar{\tau}-\delta_{2})^{2}K_{1}>0,
\end{aligned}\end{equation}
\begin{equation}\label{LLL24b}
|I(\bar{\tau},\bar{R})-I(\bar{\tau}+\delta,\bar{R})|
\leq \int_{\bar{\tau}}^{\bar{\tau}+\delta}\mu(E(r,\bar{R}))|\psi(r,\bar{R})|r^{2}dr
\leq -2\psi(\bar{\tau},\bar{R})(\bar{\tau}+\delta)^{2}\delta K_{2}.
\end{equation}
Note that $E$, $u$ and $I$ are continuous in $(r,R)$ at $(r,R)=(\bar{\tau}+\delta,\bar{R})$, we derive
\begin{equation}\label{LLL24c}
\lim_{k\rightarrow\infty}I(\bar{\tau}+\delta,R_{k})-I(\bar{\tau}+\delta,\bar{R})=0.
\end{equation}
Combining \eqref{LLL24a}, \eqref{LLL24b} and \eqref{LLL24c}, we get
$$\liminf_{k\rightarrow\infty}I(\bar{\tau}-\delta_{2},R_{k})\geq
-\frac{1}{2}\psi(\bar{\tau},\bar{R})(\delta+\delta_{2})(\bar{\tau}-\delta_{2})^{2}K_{1}
+2\psi(\bar{\tau},\bar{R})(\bar{\tau}+\delta)^{2}\delta K_{2} $$
and by letting $\delta\rightarrow0$
$$\liminf_{k\rightarrow\infty}I(\bar{\tau}-\delta_{2},R_{k})\geq
-\frac{1}{2}\psi(\bar{\tau},\bar{R})\delta_{2}(\bar{\tau}-\delta_{2})^{2}K_{1}>0.$$
This contradicts \eqref{LLL23b}, By the definition of $\tau(R_k)$
we conclude that $I(\bar{\tau}-\delta_{2},R_{k})<0$. Thus \eqref{LLL22} holds.

Combining \autoref{lem:209a} and \eqref{LLL22}, we find that the function $R\in(0,\infty)\mapsto\tau(R)$ is continuous.
\end{proof}

\begin{lem}\label{lem:209c}
The function $I=I(r,R)$ is continuous in $(r,R)$ for $R>0,\tau(R)\leq r\leq R$.
\end{lem}
\begin{proof}
  Let $\bar{R}$ and $\bar{r}$ be fixed satisfying $\tau(\bar{R})\leq\bar{r}\leq\bar{R}$.
  From \eqref{LLL21}, we obtain
  $$|I_{r}(r,R)|\leq 2K_{2}r^{2},\tau(R)\leq r\leq R,$$
  where the constant $K_{2}$ is independent of $r,R$.
  For any $\epsilon>0$, we choose a constant $\hat{\delta}$ such that
  $$16\hat{\delta} K_{2}(\bar{r}+\hat{\delta})^{2}<\epsilon.$$
  Using \autoref{lem:204} and \autoref{rmk:205}, we see $I=I(r,R)$ is continuous at $(r,R)=(\bar{r}+\hat{\delta},\bar{R})$ and there is a constant $\delta\in(0,\hat{\delta})$ such that for $R\in[\bar{R}-\delta,\bar{R}+\delta]$,
  $$2|I(\bar{r}+\hat{\delta},R)-I(\bar{r}+\hat{\delta},\bar{R})|<\epsilon.$$
  Therefore, for $r\in[\bar{r}-\delta,\bar{r}+\delta]\cap[\tau(R),R]$ and $R\in[\bar{R}-\delta,\bar{R}+\delta]$, we have
  $$\begin{aligned}
  |I(r,R)-I(\bar{r},\bar{R})|\leq&|I(r,R)-I(\bar{r}+\hat{\delta},R)|
  +|I(\bar{r}+\hat{\delta},R)-I(\bar{r}+\hat{\delta},\bar{R})|
  +|I(\bar{r}+\hat{\delta},\bar{R})-I(\bar{r},\bar{R})|\\
  \leq& 4\hat{\delta} K_{2}(\bar{r}+\hat{\delta})^{2} + |I(\bar{r}+\hat{\delta},R)-I(\bar{r}+\hat{\delta},\bar{R})| + 2\hat{\delta} K_{2}(\bar{r}+\hat{\delta})^{2}\\
  \leq& 8\hat{\delta} K_{2}(\bar{r}+\hat{\delta})^{2} + |I(\bar{r}+\hat{\delta},R)-I(\bar{r}+\hat{\delta},\bar{R})|\\
  <& \frac{1}{2}\epsilon+\frac{1}{2}\epsilon=\epsilon.
  \end{aligned}$$
  It immediately follows that $I$ is continuous at $(\bar{r},\bar{R})$.
\end{proof}

\subsection{Existence and Uniqueness for the stationary system}

\begin{lem}\label{lem:210}
  There exists a positive constant $R^{*}$ such that $\tau(R^{*})=0$ and $I(0, R^{*})=0$.
\end{lem}
\begin{proof}
  Let $R^{*}$ be defined as
  $$R^{*}=\inf\{R>0:\tau(R)>0\}.$$
  From Lemmas \ref{lem:208a} and \ref{lem:208b}, we deduce that $R^{*}$ is well-defined and is a positive finite number, and $\tau(R)=0$ for every $0<R<R^{*}$. By the continuity of $R\mapsto\tau(R)$ (see \autoref{lem:209b}), we derive $\tau(R^{*})=0$.
  From the definition of $R^{*}$, there exists a sequence $\{R^{j}\}$ with $R^{j}>R^{*}$ and $\lim_{j\rightarrow\infty}R^{j}=R^{*}$ such that
  $\tau(R^{j})>0.$
  From \autoref{lem:207}, $E(\tau(R^{j}),R^{j})=0$. It immediately follows from Lemmas \ref{lem:209c} and \ref{lem:209b} that $E(\tau(R^{*}),R^{*})=0$.
  This completes the proof.
\end{proof}

Combining Lemmas \ref{lem:210} and \ref{lem:207}, we obtain the existence of radial stationary solution.
\begin{thm}\label{thm:211}
  The radially symmetric stationary problem \eqref{sss02} admits a solution.
\end{thm}

\subsection{The uniqueness}

In this subsection we establish the uniqueness of the radially symmetric stationary solution.

We use a change of variables that transform the free boundary into a fixed boundary:
$$s=\frac{r}{R},\hspace{2ex}\tilde{E}(s,R)=E(r,R),
\hspace{2ex}\tilde{u}(s,R)=\frac{r^{2}u(r,R)}{R^{3}}.$$
Then the initial problem \eqref{LLL12E}, \eqref{LLL12U}, \eqref{LLL12U1} of $(E,u)$ is transformed to
\begin{subequations}\label{LLL28}\begin{align}
\label{LLL28E}&\tilde{u}\tilde{E}_{s}=s^{2}(-\gamma m\tilde{E}+\phi(\tilde{E})-\tilde{E}\mu(\tilde{E})\tilde{\psi}),\\
\label{LLL28U}&\tilde{u}_{s}=s^{2}\mu(\tilde{E})\tilde{\psi},\\
\label{LLL28U1}&\tilde{u}(1)=0,\tilde{E}(1)=\tilde{h}(1,R),
\end{align}\end{subequations}
where $\tilde{h}(s)=\tilde{h}(s,R)=h(sR,R)$, $h(r,R)$ is solved from \eqref{sss03Qa} and
\begin{equation}\label{LLL28psi}
\tilde{\psi}(s,R)=\sigma(sR,R)-\bar{\sigma} = \frac{\sqrt{\lambda}R}{\sinh(\sqrt{\lambda}R)} \frac{\sinh(s\sqrt{\lambda}R)}{s\sqrt{\lambda}R}-\bar{\sigma}.
\end{equation}
Set $\tilde{\tau}(R)=\tau(R)/R$.

\begin{lem}\label{lem:211}
For any $R_{2}>R_{1}>0$
\begin{equation}\label{LLL30}\tilde{u}(s,R_{2})>\tilde{u}(s,R_{1})\end{equation}
for all $1>s\geq\max\{\tilde{\tau}(R_{1}),\tilde{\tau}(R_{2})\}$.
Moreover, $\tilde{\tau}(R)$ is an increasing function when $\tilde{\tau}(R)>0$.
\end{lem}
\begin{proof}
It is obvious that $\tilde{\psi}(1,R)$, $\tilde{E}(1,R)$, $\tilde{u}_{s}(1,R)$ are independent of $R$.
One can easily compute the derivatives of $\tilde{\psi}$ as follows
\begin{equation}\label{LLL26a}\begin{aligned}
\partial_{R}\tilde{\psi}(s,R)&=
\frac{\sqrt{\lambda}\cosh(\sqrt{\lambda}R)\cosh(s\sqrt{\lambda}R)}
{s[\sinh(\sqrt{\lambda}R)]^{2}}
[s\tanh(\sqrt{\lambda}R)-\tanh(s\sqrt{\lambda}R)]<0,\\
\partial_{s}\tilde{\psi}(1,R)&=\sqrt{\lambda}R\coth(\sqrt{\lambda}R)-1>0,\\
\partial_{sR}\tilde{\psi}(1,R)
&=\frac{\sinh(2\sqrt{\lambda}R)-2\sqrt{\lambda}R}{2\sinh^{2}(\sqrt{\lambda}R)}\sqrt{\lambda}>0,
\end{aligned}\end{equation}
and calculate the derivatives of $\tilde{u}$ and $\tilde{E}$ as follows
\begin{equation}\label{LLL26b}\begin{aligned}
\tilde{u}_{s}(1,R)&=\mu(\tilde{E})\tilde{\psi}|_{s=1}>0,\\
\tilde{E}_{s}(1,R)&=
-\frac{\tilde{E}\mu(\tilde{E})\tilde{\psi}}{\mu(\tilde{E})\tilde{\psi}-Q_{E}}
\bigg|_{s=1}\partial_{s}\tilde{\psi}(1, R)<0,\\
\tilde{E}_{sR}(1, R)&= -\frac{\tilde{E}\mu(\tilde{E})\tilde{\psi}}
{\mu(\tilde{E})\tilde{\psi}-Q_{E}}\bigg|_{s=1}\partial_{sR}\tilde{\psi}(1. R)<0.
\end{aligned}\end{equation}

Set $\tilde{E}_{i}(s)=\tilde{E}(s,R_{i})$, $\tilde{\psi}_{i}(s)=\tilde{\psi}(s,R_{i})$ , $i=1,2$.
By \eqref{LLL26a} and \eqref{LLL26b}, we get
$$\tilde{E}_{2}-\tilde{E}_{1}=0,\hspace{2ex} [\tilde{E}_{2}-\tilde{E}_{1}]_{s}<0\mbox{ at }s=1$$
and hence $[\tilde{E}_{2}-\tilde{E}_{1}](s)>0$ for $0<1-s\ll1$.
Combining this with the fact that $0<\tilde{\psi}_{2}<\tilde{\psi}_{1}$, $0<1-s\ll1$
and that $\mu(E)$ is a decreasing positive function, we see that
$$\begin{aligned}
  \left[\tilde{u}_{2}-\tilde{u}_{1}\right]_{s}(s)
  =&s^{2}[\mu(\tilde{E}_{2}(s))\tilde{\psi}_{2}(s)-\mu(\tilde{E}_{1}(s))\tilde{\psi}_{1}(s)],\\
  =&s^{2}[\mu(\tilde{E}_{2}(s))-\mu(\tilde{E}_{1}(s))]\tilde{\psi}_{2}(s) + s^{2}\mu(\tilde{E}_{1}(s))[\tilde{\psi}_{2}(s)-\tilde{\psi}_{1}(s)]\\
  <0
\end{aligned}$$
and then $[\tilde{u}_{2}-\tilde{u}_{1}](s)>0$ for $0<1-s\ll1$.
Therefore,
\begin{equation}\label{LLL29} [\tilde{E}_{2}-\tilde{E}_{1}](s)>0, [\tilde{u}_{2}-\tilde{u}_{1}](s)>0\end{equation}
for $0<1-s\ll1$ and then extended to a maximal interval $(\bar{s},1)$.

We claim that $[\tilde{u}_{2}-\tilde{u}_{1}](\bar{s})>0$. Otherwise, $\tilde{u}_{1}(\bar{s})=\tilde{u}_{2}(\bar{s})\leq0$.
From \eqref{LLL28}, we obtain $(\tilde{u}\tilde{E})_{s}=(-\gamma m\tilde{E}+\phi(E))s^{2}$ and
$$\tilde{u}(s)\tilde{E}(s)=-\int_{s}^{1}[-\gamma m\tilde{E}(\xi)+\phi(\tilde{E}(\xi))]\xi^{2}d\xi.$$
Thus, by using $\phi'\leq0$, we derive
$$\begin{aligned}
0\geq&\tilde{u}_{1}(\bar{s})[\tilde{E}_{2}-\tilde{E}_{1}](\bar{s})\\
=&-\int_{\bar{s}}^{1}\{[-\gamma m\tilde{E_{2}}(\xi)+\phi(\tilde{E}_{2}(\xi))]-[-\gamma m\tilde{E}_{1}(\xi)+\phi(\tilde{E}_{1}(\xi))]\}\xi^{2}d\xi\\
>&0.
\end{aligned}$$
This yields a contradiction. In particular, $\tilde{u}_{1}(\bar{s})<\tilde{u}_{2}(\bar{s})\leq0$.

We next claim that $[\tilde{E}_{2}-\tilde{E}_{1}](\bar{s})>0$. Indeed, on the interval $(\bar{s},1)$ we have
$$\begin{aligned}
\partial_{s}\tilde{E}_{2}(s)=\frac{Q(\tilde{\psi}_{2},m,\tilde{E}_{2})s^{2}}{\tilde{u}_{2}}
<\frac{Q(\tilde{\psi}_{2},m,\tilde{E}_{2})s^{2}}{\tilde{u}_{1}}
<\frac{Q(\tilde{\psi}_{1},m,\tilde{E}_{2})s^{2}}{\tilde{u}_{1}}.
\end{aligned}$$
It immediately follows from comparison principle of ordinary differential equation that $\tilde{E}_{2}(\bar{s})>\tilde{E}_{1}(\bar{s})$.

From the definition of $\bar{s}$ and the two assertions above, we conclude that $\bar{s}=\max\{\tilde{\tau}(R_{2}),\tilde{\tau}(R_{1})\}$ and \eqref{LLL30} hold. In particular, if $\tilde{\tau}(R_{1})>0$, then by \autoref{lem:207}, $\tilde{u}(\tilde{\tau}(R_{1}),R_{1})=0$. Since $\tilde{u}(s,R_{2})<0$ for $s\in(\tilde{\tau}(R_{2}),R_{2})$, we get $\tilde{\tau}(R_{2})>\tilde{\tau}(R_{1})>0$ by using \eqref{LLL30}.
Thus, we finish the proof.
\end{proof}

As a directly consequence of \autoref{lem:211}, we have
\begin{thm}[Uniqueness]
There exists exactly one $\mathcal{R}\in(0,\infty)$ such that $\tau(\mathcal{R})=0$, $I(0,\mathcal{R})=0$, i.e., \eqref{LLL12} admit exactly one solution $(\mathcal{E},\mathcal{U},\mathcal{R})$.
\end{thm}
\begin{proof}
  If there are two constants $R_{2}>R_{1}>0$ such that
  $$\tau(R_{i})=0, I(\tau(R_{i}),R_{i})=0, i=1,2.$$
  Then $\tilde{\tau}(R_{i})=0$ and $\tilde{u}(\tilde{\tau}(R_{i}),R_{i})=0$. This contradicts \autoref{lem:211} and hence the uniqueness follows.
\end{proof}

\section{Numerical results and discussion}
In this section, we investigate numerically the asymptotic stability  of the system in the one-dimensional case.  Note that when ECM is a constant, the stationary solution is stable for small $\mu$ and unstable for large $\mu$. The uniqueness of our stationary solution indicates that our solution shall be "close" to the solution (corresponding to the constant ECM), when $\mu(E)$ is ''close'' to a constant. Hence the biological implication is that our stationary solution should be stable when $\mu(E)$ is small and unstable when $\mu(E)$ is large. However, it is a big challenge for us  to confirm our conjecture using mathematical analysis, since our system is very complex. Therefore, we perform the corresponding numerical simulations to confirm our expectation. 

Firstly, we solve the radially symmetric stationary equation \eqref{sss02}, choose 
$$Q(E,r)=-\gamma\frac{\alpha}{\beta}E+\phi(E)-E\mu(E)(\frac{\sqrt{\lambda}R}{\sinh(\sqrt{\lambda}R)}(\frac{\sinh(\sqrt{\lambda}r)}{\sqrt{\lambda}r}-\tilde{\sigma}),$$
$\phi(E)=\mu_1(1-E), \mu(E)=\frac{\mu }{1+E},$ and $R_s=R$ to be determined.  Then, we solve the system \eqref{lirui01} and compare the long time behavior of $E,\sigma, m$ with the steady state of $E, \sigma, m$.

To describe the long time behavior of our model, we simulate the time up to $t=200,$ but for most of our simulations, the profile is already very close to steady state at $t=40$, so we show mainly the dynamics for $t\leq 40$.

Here our aim is to use the different initial conditions to confirm our prediction: when $\mu(E)$ is small, the stationary solution is stable, when $\mu(E)$ is big, the stationary solution is unstable. Therefore, we divided into three cases according to the the value of parameter $\mu$. For each case, we investigate the dynamics of the density of ECM and concentration of the nutrient $\sigma$, but neglect the impact of the density of MDE; as a matter of fact, we can prove that  $m(x,t)\rightarrow \frac{\alpha}{\beta} \hspace{0.5ex} \text{uniformaly as} \hspace{0.5ex} t\rightarrow\infty$. 

Note that  $m(x,t)$ satisfies
\begin{subequations}\label{m01}\begin{align}
&\frac{\partial m(x,t)}{\partial t} =D_{m}\triangle m(x,t)+\alpha-\beta m(x,t), \hspace{-8em}&\mbox{ in }\Omega(t),\\
&\frac{\partial m(x,t)}{\partial n}=0 \hspace{-9em}&\mbox{ in }\partial\Omega(t),\\
&m(x,0)=m_0(x) \hspace{-8em}&\mbox{ in }\Omega(t).
\end{align}
\end{subequations}
We can compare $m(r,t)$ with the solution $\hat{m}(t)$ of the ODE equation
$$\frac{\partial \hat{m}(t)}{\partial t} =\alpha-\beta \hat{m}(t),\hspace{2ex} \hat{m}(0)=m_0, $$
 since $\hat{m}$ satisfies the same system \eqref{m01} as $m(x,t)$ and also satisfies 
 $$\hat{m}\rightarrow \frac{\alpha}{\beta} \hspace{2ex} \text{as} \hspace{2ex} t\rightarrow\infty,$$
 we deduce that 
 $$m(x,t)\rightarrow \frac{\alpha}{\beta} \hspace{2ex} \text{uniformaly as} \hspace{2ex} t\rightarrow\infty, $$
 by a comparison theorem for parabolic equations \cite{AFP}.
Therefore, it sufficed to consider the special case $m=\frac{\alpha}{\beta}$ from now on.

\begin{table}[ht]
\caption{parameter ranges in the model}\centering 
\begin{center}
\begin{tabular}{ c c c c}
	\hline
	Parameters & Description &Range & Reference\\
	\hline 
	$c$  & Nutrient diffusion coefficient&$10^{-5}-10^{-3}$& \cite{MAGL, ZSJU}\\ 
	$\lambda$&proliferation of nutrient &0.05-2&\cite{MAGL}\\ 
	$D_{m}$ &MDE diffusion coefficient& $10^{-3}-10$&\cite{MAGL, ZSJU}\\
	$\mu$&mobility coefficient&0.9-1.45&\cite{MAGL, ZSJU}\\
	$\mu_1$&proliferation of ECM coefficient&0.15-2.5&\cite{MAGL}\\
	$\gamma$&Rate of degradation of ECM&1-20&\cite{MAGL, ZSJU}\\
	$\alpha$&Production of MDEs&0.01-5&\cite{ZSJU}\\
	$\beta$&Decay of MDE&$10^{-1}-10$&\cite{ZSJU}\\
	\hline  
\end{tabular}
\end{center}
\end{table}

\subsection{Case I: $\mu=0.5$}
\par\vspace{1ex}\par\noindent 

\begin{figure}[H]\centering
	\includegraphics[width=0.8\textwidth]{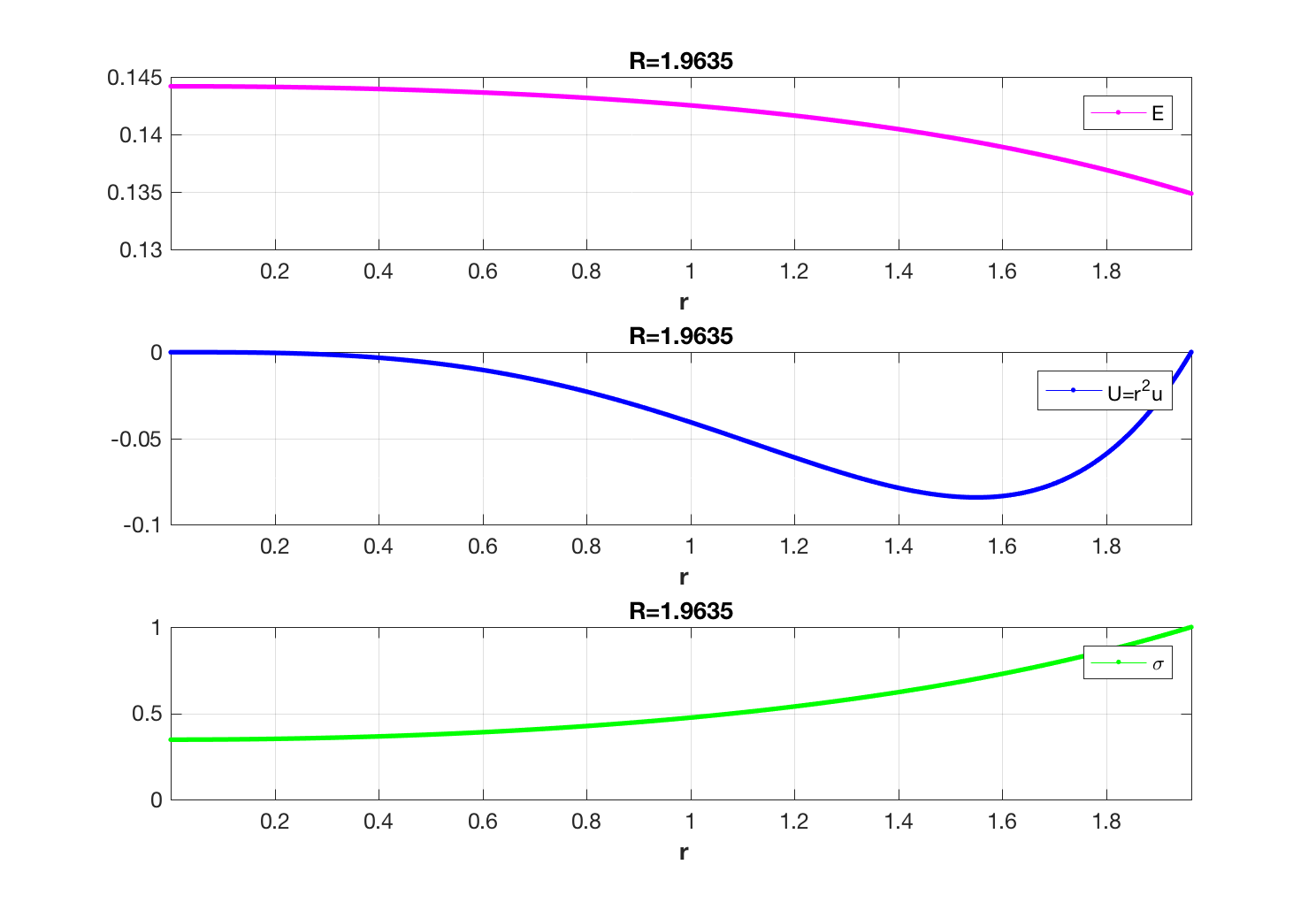}\\
	\caption{}\label{fig:01}
\end{figure}

Figure 1: One dimensional numerical results for case 1 when $\mu=0.5$. Results are snapshots of the stationary system for $E, u$ and $\sigma$ at $\gamma=10; \alpha=0.5; \beta=1; \lambda=2; \tilde{\sigma}=0.7; \mu_1=0.8;  R=1.9635; m=\frac{\alpha}{\beta}.$ The horizontal axis r indicates the spatial position, and the vertical axis indicates the density of ECM, the velocity or the concentration of nutrient listed in the legend.

 In order to show the profile of the evolution of the concentration of the nutrient $\sigma$, the density of ECM and MDE, firstly, we chose the steady state to be the initial conditions for all the system, which is a perfect case.
\begin{figure}[H]\centering
	\includegraphics[width=0.8\textwidth]{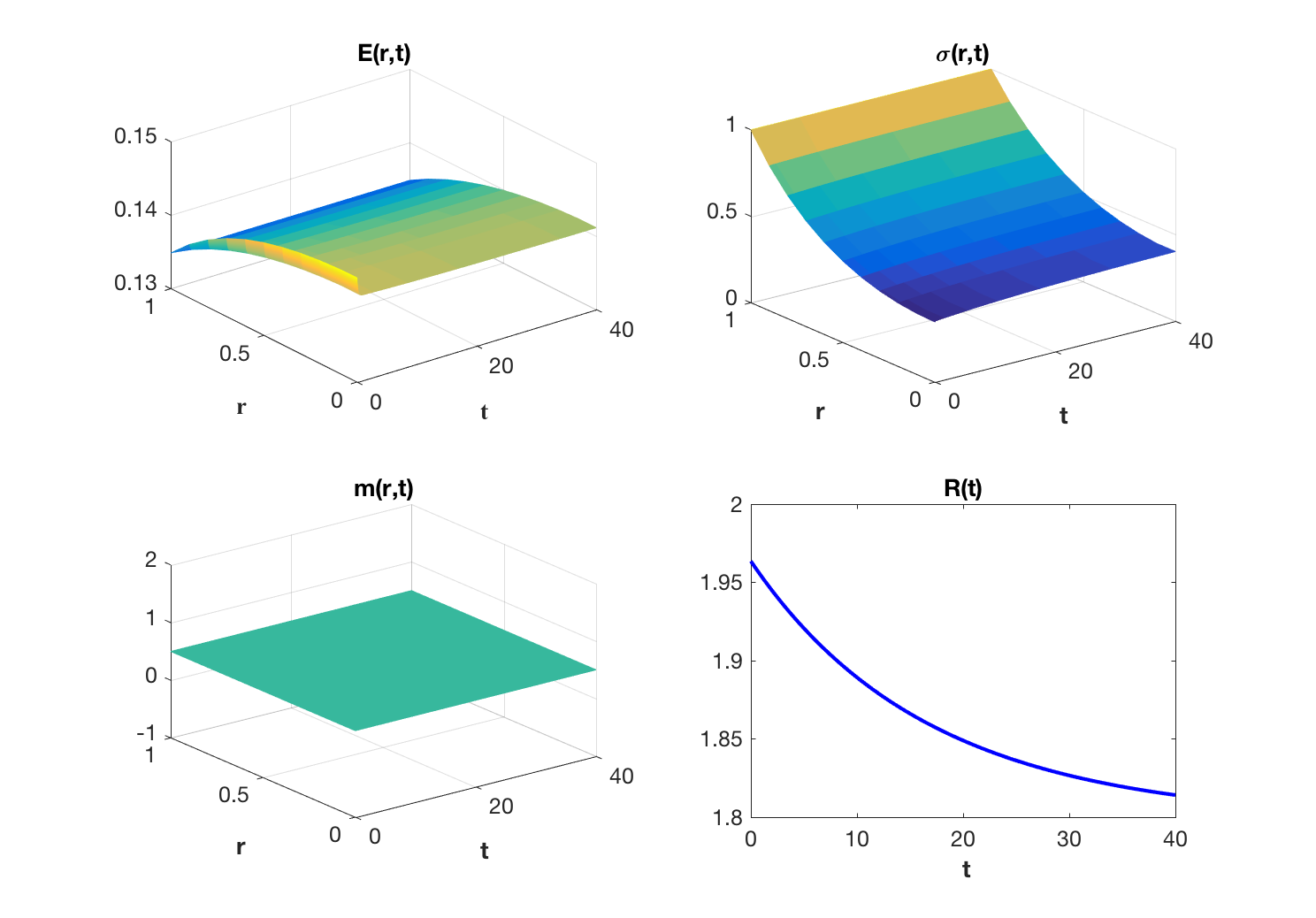}\\
	\caption{}\label{fig:0f}
\end{figure}

Next we consider the initial condition to be a small perturbation of the steady state.
\begin{figure}[H]\centering
	\includegraphics[width=0.8\textwidth]{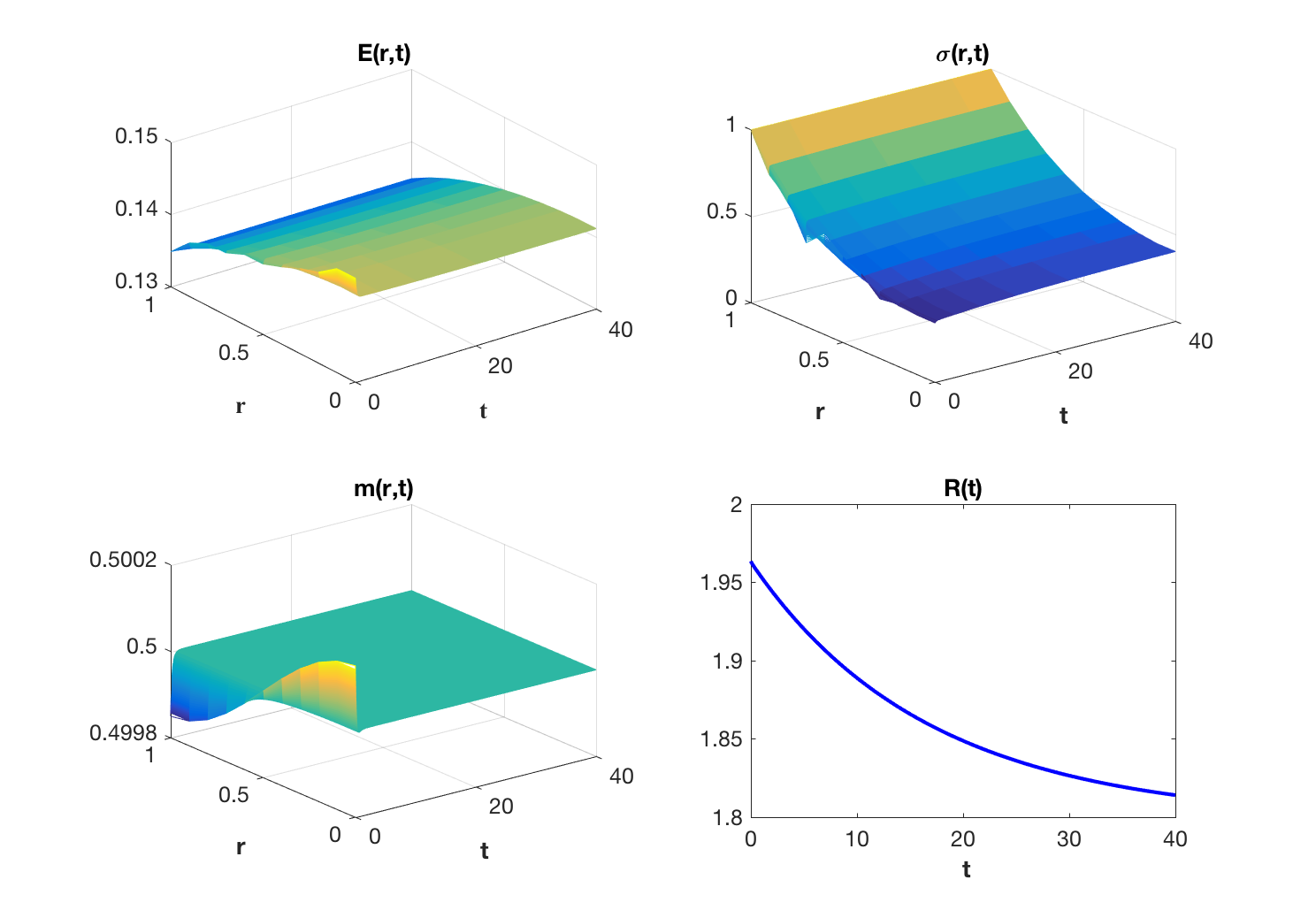}\\	
	\caption{}\label{fig:0e}
\end{figure}
 
Figure 3: In order to verify the stability of the steady state solution, we simulate a large number of initial conditions and observe the long time behavior. Under these conditions, our simulation shows, when $t\longrightarrow\infty$
$$\begin{aligned}
\sigma(r,t)\longrightarrow &\text{the steady state: } \hspace{0.5ex} \frac{\sinh{\sqrt{\lambda}r}}{r}\frac{R}{\sinh{\sqrt{\lambda}R}},\\
E(r,t)\longrightarrow &\text{the steady state: } \hspace{0.5ex} E_s,\\
m(r,t)\longrightarrow &\text{the steady state: } \hspace{0.5ex} \frac{\alpha}{\beta}.\end{aligned}$$

These plots are for the case $\mu=0.5$.

\subsection{Case II: $\mu=3.0$}
\par\vspace{1ex}\par\noindent 

\begin{figure}[H]\centering
	\includegraphics[width=0.8\textwidth]{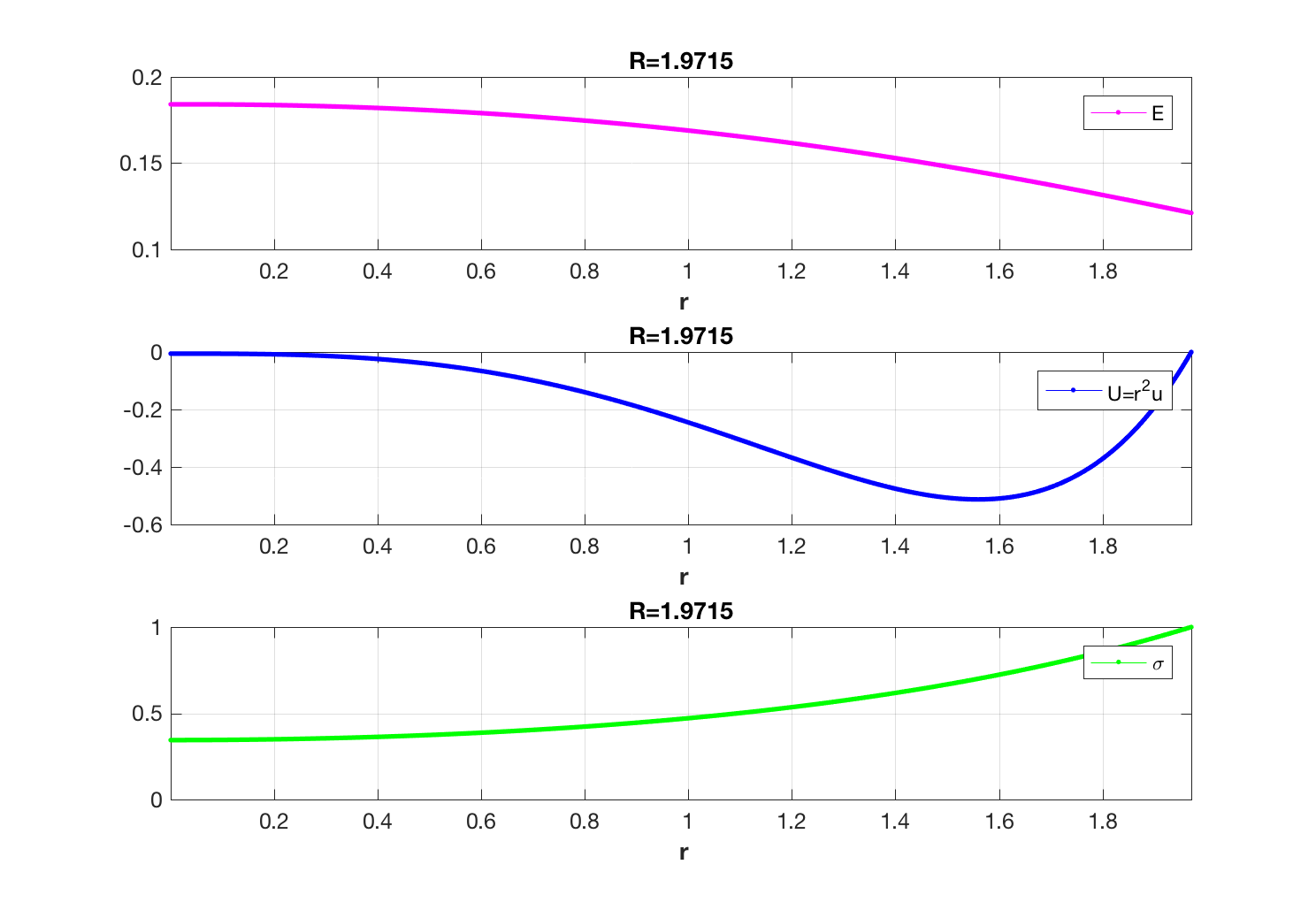}\\
	\caption{}\label{fig:01}
\end{figure}

Figure 4: One dimensional numerical results for case 2 when $\mu=3.0$. Results are snapshots of the stationary system for $E, u$ and $\sigma$ at $\gamma=10; \alpha=0.5; \beta=1; \lambda=2; \tilde{\sigma}=0.7; \mu_1=0.8;  R=1.9715; m=\frac{\alpha}{\beta}=0.5.$ The horizontal axis r indicates the spatial position, and the vertical axis indicates the density of ECM, the velocity or the concentration of nutrient listed in the legend.

\begin{figure}[H]\centering
	\includegraphics[width=0.8\textwidth]{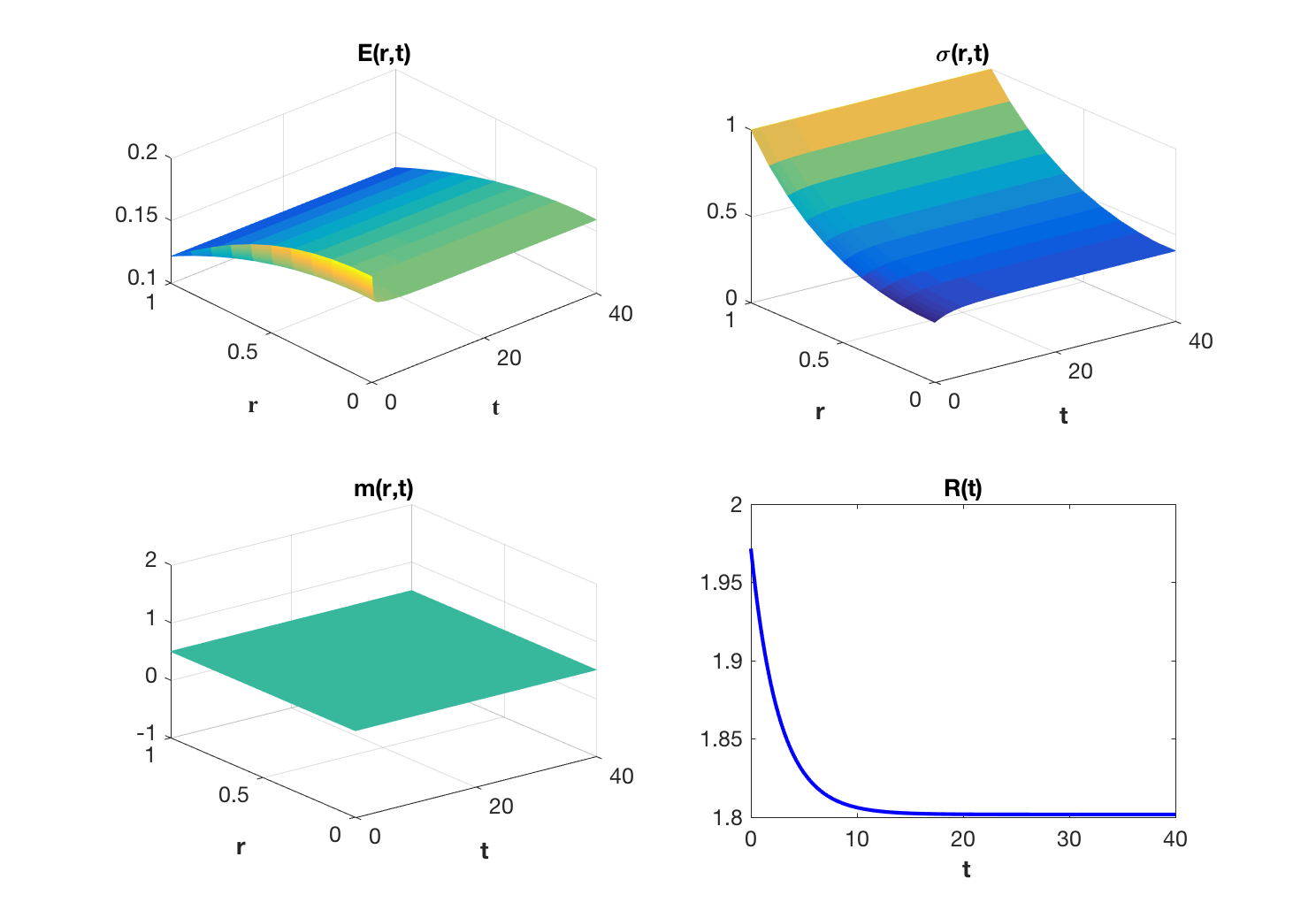}\\
	\caption{}\label{fig:0f}
\end{figure}

Figure 5: In order to show the profile of the evolution of the concentration of the nutrient $\sigma$, the density of ECM and MDE, firstly, we chose the steady state to be the initial conditions for all the system, which is a perfect case.

Next we consider the initial condition to be a small perturbation of the steady state.
\begin{figure}[H]\centering
	\includegraphics[width=0.8\textwidth]{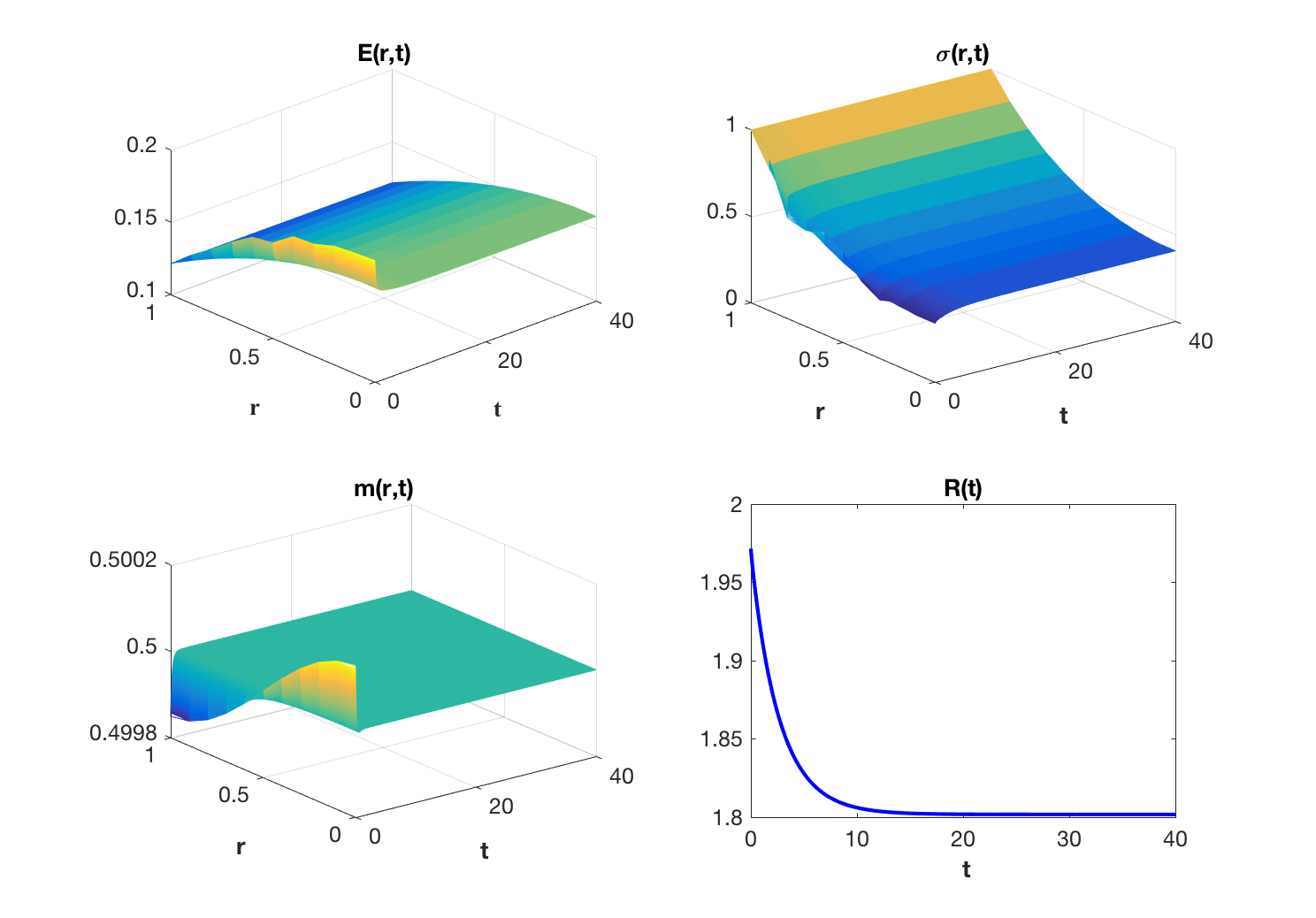}\\	
	\caption{}\label{fig:0e}
\end{figure}

Figure 6: In order to verify the stability of the steady state, we simulate a large number of initial conditions and observe the long time behavior. Under these conditions, our simulation shows, when $t\longrightarrow\infty$
$$\begin{aligned}
\sigma(r,t)\longrightarrow &\text{the steady state: } \hspace{0.5ex} \frac{\sinh{\sqrt{\lambda}r}}{r}\frac{R}{\sinh{\sqrt{\lambda}R}},\\
E(r,t)\longrightarrow &\text{the steady state: } \hspace{0.5ex} E_s,\\
m(r,t)\longrightarrow &\text{the steady state: } \hspace{0.5ex} \frac{\alpha}{\beta}.\end{aligned}$$

These plots are for the case $\mu=3.0$. 

\subsection{Case III: $\mu=10$}

\par\vspace{2ex}\par\noindent 

We now consider the parameter $\mu=10$, and obtain the critical $R\approx 1.99282$ by our simulation result. The plots indicate that the corresponding solution don't uniformly converge to the steady state. In order to show the profile of the evolution of the concentration of the nutrient $\sigma$, the density of ECM and MDE, we still consider the initial condition to be a small perturbation of the steady state. We found the steady state of $E$ is about to stay at $0.4-0.5$ when $\mu=10$, however, the $E(r,t)$ goes to a value range from $0.1-0.25$.
\begin{figure}[H]\centering
	\includegraphics[width=1.0\textwidth]{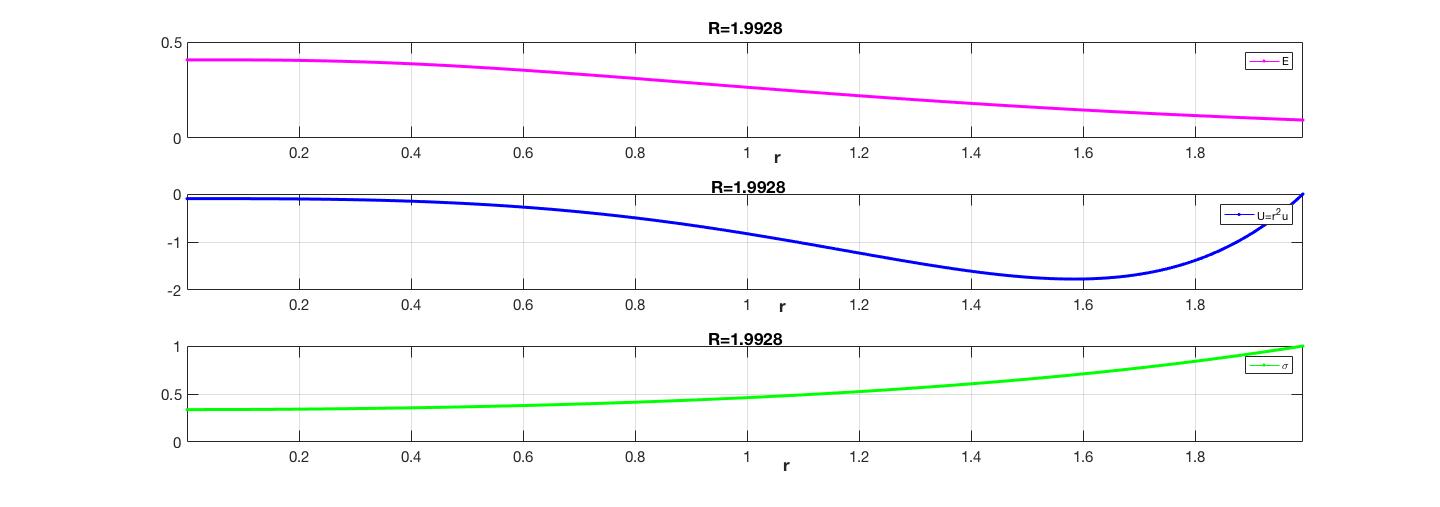}\\
	\caption{}\label{fig:0a}
\end{figure}
Figure 7: One dimensional numerical results for case 3 when $\mu=10.0$. Results are snapshots of the stationary system for $E, u$ and $\sigma$ at $\gamma=10; \alpha=0.5; \beta=1; \lambda=2; \tilde{\sigma}=0.7; \mu_1=0.8;  R=1.99282; m=\frac{\alpha}{\beta}=0.5.$ The horizontal axis r indicates the spatial position, and the vertical axis indicates the density of ECM, the velocity or the concentration of nutrient listed in the legend.

\par\vspace{2ex}\par\noindent \textbf {Appendix 1: Proof of \autoref{thm:S03A}}

\par\vspace{0ex}\par\noindent 

For simplicity, we may assume that $x_{0}=0$ and
\begin{equation}\label{eq:A02}
f(0,0)=0,\hspace{1em} g(0,0)=1,\hspace{1em}\theta=f_{x}(0,0)<1.
\end{equation}
To prove \autoref{thm:S03A}, we shall work on the solution for positive time $t>0$. The solution for negative time $t<0$ can be analyzed in exactly the same way.

\textbf{Step 1.} Existence.

In order to show the existence of solution of \eqref{eq:A1}, we shall approximate the singular equation with non-singular ones. The approximated solution $x_{\epsilon}$, $\epsilon>0$ is defined as follow:

\begin{equation}\label{eq:A04a}\begin{cases}
\displaystyle\frac{dx_{\epsilon}}{dt} = \frac{f(x_{\epsilon},t)}{\displaystyle\int_{0}^{\epsilon}g(0,\rho)d\rho+\int_{\epsilon}^{t} g(x_{\epsilon},s)ds},\hspace{2ex} r>\epsilon,\\
x_{\epsilon}(t)=0,\hspace{2ex}0\leq t\leq\epsilon.
\end{cases}\end{equation}
For $\epsilon$ sufficiently small, $\int_{0}^{\epsilon}g(0,\rho)d\rho>0$ and the denominator of the right hand side of $dx_{\epsilon}/dr$ does not vanish at $r=\epsilon$. For such an $\epsilon$, the equation \eqref{eq:A04a} does not have singularity. If one denotes that denominator by $y=y(t)$, then $x=x(t), y=y(t)$ satisfies ordinary differential equations. By the standard ODE theory, the solution $x_{\epsilon}$ of \eqref{eq:A04a} exists and is unique, and the maximal existence interval is denoted by $[0,T_{\epsilon})$ for $0<\epsilon\ll1$. Moreover, $x_{\epsilon}$ satisfies
\begin{equation}\label{eq:A04b}
x_{\epsilon}(t)=
\int_{0}^{t}\frac{f(x_{\epsilon}(\tau)),\tau)}{\int_{0}^{\tau}g(x_{\epsilon}(s),s)ds}d\tau,
t>\epsilon.
\end{equation}
 	
We first estimate the lower bound and the upper bound of $x_{\epsilon}$.

Let $\eta>0$ be fixed such that $2\eta<1-\theta$ and $\eta<1$. By the continuity, there exists a $\delta=\delta(\eta)>0$ such that
\begin{equation}\label{eq:A05}
|g(x,t)-1|<\eta,\hspace{0.5em} |f_{x}(x,t)-\theta|<\eta,\hspace{0.5em} |f(x,t)|<\eta
\end{equation}
for all $(x,t)\in\bar{D}_{\delta},$ where $\bar{D}_{\delta}=\{(x,t):|x|\leq\delta, |t|\leq\delta\}$. Let $M>0$ be such that
$$|f_{x}(x,t)|<M, |f_{t}(x,t)|<M\mbox{ for }(x,t)\in\bar{D}_{\delta}.$$
Now we take two positive constants $\kappa$ and $\mathcal{T}$ such that
\begin{equation}\label{eq:A06}
\kappa>\frac{M}{1-\theta-2\eta},\mathcal{T}=\min\{\delta,\frac{\delta}{\kappa}\}.
\end{equation}
We claim that $x_{\epsilon}$ exists for $t\in[0,\mathcal{T}]$ for every $\epsilon\in(0,\mathcal{T})$,
\begin{equation}\label{eq:A07a}
-\kappa t<x_{\epsilon}(t)<\kappa t,~~0\leq t\leq\mathcal{T},
\end{equation}
\begin{equation}\label{eq:A07b}
|x_{\epsilon}'(t)|\leq\frac{M\kappa+M}{1-\eta},\epsilon<t\leq\mathcal{T}.
\end{equation}
Indeed, we have for $0<t\leq\mathcal{T}$,
$$\begin{aligned}
\kappa(1-\eta)t-f(\kappa t,t)&>\kappa(1-\eta) t-(\theta+\eta)\kappa t - Mt\\
&= [(1-\theta-2\eta)\kappa-M] t> 0,\\
-\kappa(1-\eta)t-f(-\kappa t,t)&< -\kappa(1-\eta)t + (\theta+\eta)\kappa t + Mt\\
&= -[(1-\theta-2\eta)\kappa-M] t<0.
\end{aligned}$$
Thus, $\overline{x}(t)=\kappa t$ and $\underline{x}(t)=-\kappa t$ satisfy
$$\overline{x}'(t)>\frac{f(\overline{x}(t),t)}{(1-\eta)t},
~~\underline{x}'(t)<\frac{f(\underline{x}(t),t)}{(1-\eta)t},
~~0<t\leq\mathcal{T}.$$
Now we proceed to prove \eqref{eq:A07a}. If there is a first $\bar{t}\in(\epsilon,\mathcal{T}]\cap(\epsilon,T_{\epsilon})$ such that
$$x_{\epsilon}(t)<\overline{x}(t), t\in(\epsilon,\bar{t})\mbox{ and }x_{\epsilon}(\bar{t})=\overline{x}(\bar{t}),$$
then, by the definition of derivatives, $x_{\epsilon}'(\bar{t})\geq\overline{x}'(\bar{t})=\kappa>0$, and $f(x_{\epsilon}(\bar{t}),\bar{t})>0$ by the equation of $x_{\epsilon}$.
However, at $t=\bar{t}$,
$$x_{\epsilon}'(\bar{t})
=\frac{f(x_{\epsilon}(\bar{t}),\bar{t})}{\int_{0}^{\bar{t}}g(x(s),s)ds}
<\frac{f(\overline{x}(\bar{t}),\bar{t})}{(1-\eta)\bar{t}}
<\overline{x}'(\bar{t}).$$
This yields a contradiction. Thus, $x_{\epsilon}(t)<\overline{x}(t)$, $t\in(\epsilon,\mathcal{T}]\cap(\epsilon,T_{\epsilon})$. Similarly, we can prove $x_{\epsilon}>\underline{x}(t)$, $t\in(\epsilon,\mathcal{T}]\cap(\epsilon,T_{\epsilon})$. Since $T_{\epsilon}$ is the maximal existence time of $x_{\epsilon}$, we see $T_{\epsilon}>\mathcal{T}$. Therefore, \eqref{eq:A07a} and then \eqref{eq:A07b} are established.

Now we show the existence of solution of \eqref{eq:A1}. Indeed, from \eqref{eq:A07a} and \eqref{eq:A07b}, we see that $x_{\epsilon}$ is uniformly bounded in the Lipschitz function space $\mbox{Lip}[0,\mathcal{T}]$. By passing to a subsequence if necessary, we may assume that
$$x_{\epsilon}\rightarrow x\mbox{ in }C^{0}[0,\mathcal{T}]$$
as $\epsilon\downarrow0$ for some function $x\in\mbox{Lip}[0,\mathcal{T}]$. From \eqref{eq:A04b}, we obtain $x\in\mbox{Lip}[0,\mathcal{T}]$ satisfies
$$ x(t)=\int_{0}^{t}\frac{f(x(\tau),\tau)}{\int_{0}^{\tau}g(x(s),s)ds}d\tau. $$
Thus, $x\in\mbox{Lip}[0,\mathcal{T}]\cap C^{1}(0,\mathcal{T}]$ is a solution to \eqref{eq:A1}. Moreover, $x$ is of class $C^{1}[0,\mathcal{T}]$ and $x'(0)=k=\gamma/(1-\theta)$, $\gamma=f_{t}(0,0)$; this is an application of \autoref{prop:FH2008} below, which is a corrected version of Lemma 9.3 of \cite{FriedmanHu-08}.

\textbf{Step 2.} Uniqueness when $\theta\in(-1,1)$.

Assume that $x(t), y(t)$ are two $C^{1}$ solutions of \eqref{eq:A1} and they are well-defined at a common interval $[0,T]$. Set
\begin{equation}\label{eq:A11} L=\max\{\max_{t\in[0,T]}|x'(t)|, \max_{t\in[0,T]}|y'(t)|\}.\end{equation}
Let us fix $\eta\in(0,1)$ satisfying
\begin{equation}\label{eq:A12}\tilde{\theta}=\frac{|\theta| +\eta}{1-\eta}\in(0,1).\end{equation}
By the continuity, there exists a $\delta=\delta(\eta)>0$ such that \eqref{eq:A05} holds.
Additionally, let $M=M_{\delta}>0$ satisfy
\begin{equation}\label{eq:A13} |f_{x}(x,t)|<M, |f_{t}(x,t)|<M, |g(x,t)-g(y,t)|\leq M|x-y|\end{equation}
for all $(x,t)$ and $(y,t)$ in $\bar{D}_{\delta}$.
Thus, $z(t)=x(t)-y(t)$ satisfies
$$\frac{dz(t)}{dt} = \frac{f(x(t),t)-f(y(t),t)}{\int_{0}^{t}g(x(s),s)ds} - \frac{f(y(t),t)\int_{0}^{t}[g(x(s),s)-g(y(s),s)]ds} {\int_{0}^{t}g(x(s),s)ds\int_{0}^{t}g(y(s),s)ds},$$
it follows from \eqref{eq:A05} and \eqref{eq:A13} that
$$|z'(t)| \leq \frac{(|\theta|+\eta)|z(t)|}{(1-\eta)t}+\frac{M(L+1)t\cdot t M}{(1-\eta)^{2}t^{2}} \max_{s\in[0,t]}|z(s)|,$$
for $0<t\leq T_{1}=\min\{\delta,\frac{\delta}{L}\}$.
If we set $\tilde{M}=M^{2}(L+1)(1-\eta)^{-2}$ and $Z(t)=\max_{s\in[0,t]}|z(s)|$, then
\begin{equation}\label{eq:A14}
|z'(s)|\leq\frac{\tilde{\theta}|z(s)|}{s}+\tilde{M}Z(t),
\end{equation}
for all $s\in(0,t]$ and all $t\in(0,T_{1}]$.

We claim that
\begin{equation}\label{eq:A15}
0\leq |z(s)|\leq\frac{\tilde{M}s}{1-\tilde{\theta}}Z(t), s\in(0,t],t\in(0,T_{1}].
\end{equation}
Indeed, from \eqref{eq:A11} and \eqref{eq:A14}, one easily obtain by induction that
$$|z'(s)|\leq a_{k}, |z(s)|\leq a_{k} t, s\in[0,t].$$
Here $a_{0}=2L$ and $a_{k+1}=\tilde{\theta}a_{k}+\tilde{M}Z(t)$. \eqref{eq:A15} follows by the following fact
$$\lim_{k\to\infty}a_{k}=\frac{\tilde{M}}{1-\tilde{\theta}}Z(t).$$

Take $t=t^{*}\in(0,T_{1}]$ such that $\tilde{M}t^{*}<1-\tilde{\theta}$, then
$$0\leq Z(t)\leq\frac{\tilde{M}t^{*}}{1-\tilde{\theta}}Z(t), t\in[0,t^{*}].$$
It immediately follows that $Z(t^{*})=0$ and hence $z(t)=0$, $0\leq t\leq t^{*}.$ This shows the uniqueness in the case of $\theta\in(-1,1)$.

\textbf{Step 3.} Uniqueness when $\theta\in(-\infty,0)$.

Assume that $x(t), y(t)$ are two solutions of \eqref{eq:A1} and they are well-defined at a common interval $[0,T]$.
Then $z(t)=x(t)-y(t)$ satisfies
$$(z(t)t^{-\theta})'=t^{-\theta}[z'(t)-\frac{\theta z(t)}{t}]=:t^{-\theta} A(t),$$
where
$$\begin{aligned}
A(t)=&\left(\frac{f(x(t),t)}{\int_{0}^{t}g(x(s),s)ds}-\frac{\theta x(t)}{t}\right)-\left(\frac{f(y(t),t)}{\int_{0}^{t}g(y(s),s)ds}-\frac{\theta y(t)}{t}\right)\\
=&\frac{[f(x(t),t)-\theta x(t)]-[f(y(t),t)-\theta y(t)]}{\int_{0}^{t}g(x(s),s)ds}\\
&+\frac{\theta [x(t)-y(t)]\int_{0}^{t}[1-g(x(s),s)]ds}{t\int_{0}^{t}g(x(s),s)ds}\\
&-\frac{f(y(t),t)\int_{0}^{t}[g(x(s),s)-g(y(s),s)]ds}
{\int_{0}^{t}g(x(s),s)ds\int_{0}^{t}g(y(s),s)ds}.\\
\end{aligned}$$
Let $\delta_{0}\in(0,T)$ and $M_{0}>0$ such that
\begin{equation}\label{eq:A17a} |g(x,t)-g(y,t)|\leq M_{0}|x-y| \end{equation}
for all $(x,t)$ and $(y,t)$ in $\bar{D}_{\delta_{0}}$.
Let us fix $\eta>0$ and $\delta\in(0,\delta_{0})>0$ satisfying
\begin{equation}\label{eq:A17b}
\alpha\in(0,1)\mbox{ with } \alpha=\frac{1}{|\theta|}\left(\frac{\eta}{1-\eta}
+\frac{|\theta|\eta}{(1-\eta)^{2}}+\frac{M_{1}\eta}{(1-\eta)^{2}}\right)
\end{equation}
and that \eqref{eq:A05} holds.
Now let $t^{*}\in(0,\delta)$ such that
$$|x(t)|<\delta, |y(t)|<\delta, t\in[0,t^{*}].$$
Using \eqref{eq:A05}, \eqref{eq:A17a} and \eqref{eq:A17b}, we see that for $t\in(0,t^{*}]$,
$$\begin{aligned}
|A(t)|
\leq\frac{\eta|z(t)|}{(1-\eta)t}+\frac{|\theta||z(t)|\cdot t\eta}{t^{2}(1-\eta)}+\frac{\eta\cdot M_{1}Z(t)t}{(1-\eta)^{2}t^{2}}
\leq\alpha|\theta|\frac{Z(t)}{t},
\end{aligned}$$
where $Z(t)=\max_{s\in[0,t]}|z(s)|$ is a nondecreasing, nonnegative function. Thus,
$$|(z(t)t^{-\theta})'|\leq-\theta\alpha t^{-\theta-1}Z(t),\hspace{2ex} t\in(0,t^{*}].$$
Integrating the inequality above, we get
$$|z(t)t^{-\theta}|\leq -\theta\alpha Z(t)\int_{0}^{t}s^{-\theta-1}ds=\alpha Z(t)t^{-\theta}.$$
It immediately follows that $0\leq Z(t)\leq\alpha Z(t)$, $t\in[0,t^{*}]$. Hence $Z(t)=0$, $x(t)-y(t)=z(t)=0$, $t\in[0,t^{*}]$.
This shows the uniqueness in the case of $\theta\in(-\infty,0)$.

\par\vspace*{2ex}\par\noindent\textbf{Appendix 2: Proof of \autoref{thm:S03B}: }

\par\vspace{0ex}\par\noindent 

Without loss of generality, we may assume that
$$g(\varphi(\mu),0,\mu)\equiv1\mbox{ and }f_{x}(\varphi(\mu),0,\mu)=\theta(\mu).$$

\textbf{Step 1.} We prove that there exists two constants $\mathcal{T}>0$ and $\epsilon>0$ such that

(i) For every $\mu\in B_{\mu_{0},\epsilon}$, the solution $x=x(t,\mu)$ of \eqref{eq:A2} exists in $[0,\mathcal{T}]$;

(ii) The function $(t,\mu)\mapsto x(t,\mu)$ is continuous in $[0,\mathcal{T}]\times B_{\mu_{0},\epsilon}$.

Since $\theta_{0}=\theta(0)<1$, $\theta(\mu)$ is continuous for $\mu$ in a neighborhood of $\mu_{0}$, there exists a $\delta_{1}$ such that $\theta(\mu)<1$ for $|\mu-\mu_{0}|\leq\delta_{1}$.
Let us fix $\eta\in(0,1)$ satisfying
$$\inf\{1-\theta(\mu)-2\eta:\mu\in B_{\mu_{0},\delta_{1}}\}>0.$$
For such an $\eta$, there exists a $\delta_{2}\in(0,\delta_{1})$ and $M>0$ such that
$$|f_{x}(x,t,\mu)-\theta(\mu)|<\eta,\hspace{0.5em}|g(x,t,\mu)-1|<\eta, $$
$$|f_{x}(x,t,\mu)|<M,\hspace{0.5em}|f_{x}(x,t,\mu)|<M $$
for all $(x,t,\mu)$ satisfying $|x-\varphi(\mu)|<\delta_{2}, |t|<\delta_{2},\mu\in B_{\mu_{0},\delta_{2}}$.
Now we fix $\kappa$ and $\mathcal{T}>0$ such that
$$\kappa>\max\{\frac{M}{1-\theta(\mu)-2\eta}:\mu\in B_{\mu_{0},\delta_{2}}\}\mbox{ and } \mathcal{T}=\min\{\delta_{2},\frac{\delta_{2}}{\kappa}\}.$$
From \autoref{thm:S03A}, the solution $x(t,\mu)$ of \eqref{eq:A2} exists and is unique,
the (right) maximal existence interval is denoted by $[0,T_{\mu})$.
From the proof of the existence results (see step 1 in the proof of \autoref{thm:S03A}), we see that
$$T_{\mu}>\mathcal{T},\hspace{0.5em}|x(t,\mu)-\varphi(\mu)|<\kappa t,\hspace{0.5em}
\left|\frac{\partial x(t,\mu)}{\partial t}\right|\leq\frac{M\kappa + M}{1-\eta}$$
for all $t\in[0,\mathcal{T}]$ and $\mu\in B_{\mu_{0},\delta_{2}}$.
Thus $\{x(\cdot,\mu):\mu\in B_{\mu_{0},\delta_{2}}\}$ is a bounded in $C^{1}[0,\mathcal{T}]$. By the Ascoli-Arzel\'{a} theorem, $\{x(\cdot,\mu):\mu\in B_{\mu_{0},\delta_{2}}\}$ is a pre-compact subset in $C[0,\mathcal{T}]$.
Then for any $\bar{\mu}\in B_{\mu_{0},\delta_{2}}$ and convergent sequence $\{\mu_{k}\}$ with limit $\bar{\mu}$, there exists a further subsequence $\{\mu_{k_{l}}\}$ such that the corresponding subsequence of function $\{x(\cdot,\mu_{k_{l}})\}$ is a convergent subsequence in $C[0,\mathcal{T}]$, the limit is denoted $\bar{x}$, which belongs to $\mbox{Lip}[0,\mathcal{T}]$. Note that the solution $x(t,\mu)$ satisfies
\begin{equation}\label{eq:A22}
x(t,\mu) = \varphi(\mu)+\int_{0}^{t}\frac{f(x(t,\mu)),t,\mu)}{\int_{0}^{t}g(x(s,\mu),s,\mu)ds},t>0.
\end{equation}
We know the limit function $\bar{x}\in\mbox{Lip}[0,\mathcal{T}]$ also satisfies \eqref{eq:A22} with $\mu=\bar{\mu}$. Hence $\bar{x}$ is the unique solution of \eqref{eq:A2} with $\mu=\bar{\mu}$, the limit function $\bar{x}=x(\cdot,\bar{\mu})$ depends only on $\bar{\mu}$, which is independent of the choice of subsequence of $\{\mu_{k}\}$. Thus, the full sequence $\{x(\cdot,\mu_{k})\}$ converges to $x(\cdot,\bar{\mu})$ in $C[0,\mathcal{T}]$. Therefore, $(t,\mu)\mapsto x(t,\mu)$ is continuous in $[0,\mathcal{T}]\times B_{\mu_{0},\delta_{2}}$.

\textbf{Step 2.} We complete the proof.
We rewrite our equation in the form
\begin{equation}\label{eq:A23a}\begin{cases}
\frac{dx}{dt}=\frac{f(x,t,\mu)}{y},\\
\frac{dy}{dt}=g(x,t,\mu),\\
x(0)=0,y(0)=0.
\end{cases}\end{equation}
\autoref{thm:S03A} and step 1 suggest that the $C^{1}$ solution $(x(t,\mu),y(t,\mu))$ of \eqref{eq:A23a} exists and is unique in a time interval $[0,\mathcal{T}]$ for $\mu\in B_{\mu_{0},\delta_{2}}$, and the solution is continuous in $(t,\mu)$ for $t\in[0,\mathcal{T}]$, $\mu\in B_{\mu_{0},\delta_{2}}$, $y(t,\mu)>(1-\eta)t$, $(t,\mu)\in(0,\mathcal{T}]\times B_{\mu_{0},\delta_{2}}$. Moreover, when $\mu=\mu_{0}$, the solution $(x(t,\mu_{0}),y(t,\mu_{0}))$ exists in a bounded closed interval $[0,b]$.

In order to show the continuous dependence in $\mu$ beyond time $\tau$, we consider the following initial problem
\begin{equation}\label{eq:A23b}\begin{cases}
\frac{d\hat{x}}{dt}=\frac{f(\hat{x},t,\mu)}{y},\\
\frac{d\hat{y}}{dt}=g(\hat{x},t,\mu),\\
\hat{x}(\mathcal{T})=\xi,\hat{y}(\mathcal{T})=\zeta.
\end{cases}\end{equation}
Set $\xi_{0}=x(\mathcal{T},\mu_{0})$, $\zeta_{0}=y(\mathcal{T},\mu_{0})$. The solution of \eqref{eq:A23b} exists and is denoted by $\hat{x}=\hat{x}(t,\xi,\zeta,\mu),\hat{y}=\hat{y}(t,\xi,\zeta,\mu)$ for at least $(\xi,\zeta,\mu)$ in a neighborhood of $(\xi_{0},\zeta_{0},\mu_{0})$.
By the classical ODE theory, solutions of \eqref{eq:A23b} and \eqref{eq:A23a} are the same when $\xi=x(\mathcal{T},\mu)$, $\xi=y(\mathcal{T},\mu)$.
In particular, when $(\xi,\zeta,\mu)=(\xi_{0},\zeta_{0},\mu_{0})$, the solution $(\hat{x},\hat{y})$ exists in a bounded closed interval $[\mathcal{T},b]$.
From the classical ODE theory, there exists a $\delta_{3}\in(0,\delta_{2})$ such that
(i) the solution $\hat{x}=\hat{x}(t,\xi,\zeta,\mu),\hat{y}=\hat{y}(t,\xi,\zeta,\mu)$ of \eqref{eq:A23b} exist in $[\mathcal{T},b]$ for every $|\xi-\xi_{0}|<\delta_{3}$, $|\zeta-\zeta_{0}|<\delta_{3}$ and $\mu\in B_{\mu_{0},\delta_{3}}$;
(ii) $\hat{x}=\hat{x}(t,\xi,\zeta,\mu),\hat{y}=\hat{y}(t,\xi,\zeta,\mu)$ are continuous functions in $(t,\xi,\zeta,\mu)\in [\mathcal{T},b]\times [\xi_{0}-\delta_{3},\xi_{0}+\delta_{3}]\times[\zeta_{0}-\delta_{3},\zeta_{0}+\delta_{3}]\times B_{\mu_{0},\delta_{3}}$.
By step 1, we know there exists a $\delta_{4}\in(0,\delta_{3})$ such that
$$|x(\mathcal{T},\mu)-x(\mathcal{T},\mu_{0})|<\delta_{3}, |y(\mathcal{T},\mu)-y(\mathcal{T},\mu_{0})|<\delta_{3},\mu\in B_{\mu_{0},\delta_{4}}.$$
Thus, $x(t,\mu)=\hat{x}(t,x(\mathcal{T},\mu),y(\mathcal{T},\mu),\mu)$ and $y(t,\mu)=\hat{y}(t,x(\mathcal{T},\mu),y(\mathcal{T},\mu),\mu)$ exist in the time interval $[\mathcal{T},b]$ for every $\mu\in B_{\mu_{0},\delta_{4}}$, and these functions are continuous in $(t,\mu)$ for $t\in[\mathcal{T},b]$ and $\mu\in B_{\mu_{0},\delta_{4}}$.
Therefore, the solution $x(t,\mu)$ of \eqref{eq:A2} exists in $[0,b]$ for $\mu\in B_{\mu_{0},\delta_{4}}$, and $x(t,\mu)$ is continuous in $(t,\mu)\in[0,b]\times B_{\mu_{0},\delta_{4}}$.

Similarly, the solution $x(t,\mu)$ of \eqref{eq:A2} exists in $[a,0]$ for $\mu\in B_{\mu_{0},\delta_{5}}$, and $x(t,\mu)$ is continuous in $(t,\mu)\in[a,0]\times B_{\mu_{0},\delta_{5}}$ for some $\delta_{5}$.
\autoref{thm:S03B} is proved by taking $\delta=\min\{\delta_{4},\delta_{5}\}$.

\par\vspace{2ex}\par\noindent \textbf{Appendix 3: Fixing an error in \cite{FriedmanHu-08}}

\par\vspace{0ex}\par\noindent 

There is an important reference \cite{FriedmanHu-08}, which is also a singular equation but without the nonlocal integral term. However, we found an error in the proof of Lemma 9.3 in \cite{FriedmanHu-08} and correct it as follows:
\begin{lem}[\cite{FriedmanHu-08}]\label{lem:FH2008}
Let $\phi\in C[0,a]\cap C^{1}(0,a]$ be a solution of
$$v(r)\frac{d\phi}{dr}=F(\phi,r),$$
where $v\in C^{1}[0,a]$ and $F\in C^{1}$ in an open set containing the image $\{(\phi(r),r):0\leq r\leq a\}$ and satisfies $v(0)=0,F(\phi(0),0)=0$ and
\begin{equation}\label{eq:FH01}F_{\phi}(\phi(0),0)\neq v'(0),\hspace{2ex} v'(0)\neq0.\end{equation}
Then we have that $\phi\in C^{1}[0,a]$ and
\begin{equation}\label{eq:FH0002}
\lim_{r\downarrow0}\frac{d\phi}{dr}(r)
=\frac{d\phi}{dr}(0)=\frac{F_{r}(\phi(0),0)}{v'(0)-F_{\phi}(\phi(0),0)}.
\end{equation}
\end{lem}

This is Lemma 9.3 in \cite{FriedmanHu-08}.

Consider the following equation
\begin{equation}\label{eq:FH02}\begin{cases}
 \frac{d\phi}{dr}=\frac{F(\phi,r)}{r}, r>0,\\
 \phi(0)=0,
\end{cases}\end{equation}
where $F$ is a smooth function with $F(0,0)=0$.

\emph{Example 1.} $F(\phi,r)=\alpha\phi$ with $\alpha\in(0,1)$.
Then all the continuous solution of \eqref{eq:FH02} are given by
$$\phi(r)=cr^{\alpha},r\geq0.$$

\emph{Example 2.} $F(\phi,r)=\phi^{2}$.
Then \eqref{eq:FH02} has exactly two kind of continuous solution: $\phi(r)\equiv0$ and
$$\phi(r)=\frac{1}{\ln(c/r)}, 0<r<c$$
These two examples satisfy the condition \eqref{eq:FH01}, while the continuous solutions of \eqref{eq:FH02} is not unique and is not differentiable at $r=0$ except the trial solution $\phi(r)\equiv0$.

Thus, there is an error in \autoref{lem:FH2008} when $F_{\phi}(\phi(0),0)/v'(0)\in[0,1)$, $v'(0)\neq0$.

\begin{prop}\label{prop:FH2008}
  Let the assumptions in \autoref{lem:FH2008} hold. Set $\gamma=F_{\phi}(\phi(0),0)/v'(0)$. Then

  (i) If $\gamma\not\in[0,1]$, then $\phi\in C^{1}[0,a]$ and \eqref{eq:FH0002} holds.

  (ii) If $\gamma\in[0,1)$ and $\phi\in Lip[0,a]$, then $\phi\in C^{1}[0,a]$ and \eqref{eq:FH0002} holds.
\end{prop}

\begin{proof}
  As in the appendix in \cite{FriedmanHu-08}, we rewrite our equation in the linear form as follows
  $$\frac{d\phi}{dr}=\alpha_{1}(r)(\phi-\phi(0))+\alpha_{2}(r),$$
  where
  $$\alpha_{1}(r)
  =\frac{F(\phi(r),r)-F(\phi(0),r)}{(\phi(r)-\phi(0))v(r)}
  =\frac{F_{\phi}(\phi(0),0)}{v'(0)r}(1+o(1)),$$
  $$\alpha_{2}(r)
  =\frac{F(\phi(0),r)-F(\phi(0),0)}{v(r)}
  =\frac{F_{r}(\phi(0),0)}{v'(0)}(1+o(1)),$$
  we then have
  \begin{equation}\label{eq:FH05}
  \frac{d}{dr}\left\{(\phi(r)-\phi(0))G(r)\right\}=\alpha_{2}(r)G(r)
  \end{equation}
  with
  $$G(r)=\exp\left(\int_{r}^{a}\alpha_{1}(s)ds\right),
  \hspace{2ex} G'(r)=-\alpha_{1}(r)G(r).$$

  \emph{Case 1.} $F_{\phi}(\phi(0),0)/v'(0)>1, v'(0)\neq0$.
  This case is proved in the appendix of \cite{FriedmanHu-08}.

  \emph{Case 2.} $F_{\phi}(\phi(0),0)/v'(0)<0, v'(0)\neq0$. In this case, there exist $\epsilon>0$ and $r_{\epsilon}\in(0,a)$ such that $r\alpha_{1}(r)<-\epsilon$ for $0<r<r_{\epsilon}$. Then
  $$0<G(r)<C_{\epsilon}r^{\epsilon},\hspace{2ex} r\in(0,a]$$
  for some $C_{\epsilon}>0$.
  We integrate \eqref{eq:FH05} over $[0,r]$ to obtain
  \begin{equation}\label{eq:FH06}
  (\phi(r)-\phi(0))G(r)=\int_{0}^{r}\alpha_{2}(\eta)G(\eta)d\eta.
  \end{equation}
  By L'H\^{o}spital's rule,
  \begin{equation}\label{eq:FH07a}\begin{aligned}
    \lim_{r\downarrow0}\frac{\phi(r)-\phi(0)}{r}
    =&\lim_{r\downarrow0}\frac{\int_{0}^{r}\alpha_{2}(\eta)G(\eta)d\eta}{rG(r)}\\
    =&\lim_{r\downarrow0}\frac{\alpha_{2}(r)G(r)}{[1-r\alpha_{1}(r)]G(r)}\\
    =&\lim_{r\downarrow0}\frac{\alpha_{2}(r)}{1-r\alpha_{1}(r)}
    =\frac{F_{r}(\phi(0),0)}{v_{r}(0)-F_{\phi}(\phi(0),0)}
  \end{aligned}\end{equation}
  and
  \begin{equation}\label{eq:FH07b}\begin{aligned}
    \lim_{r\downarrow0}\frac{d\phi(r)}{dr}
    =&\lim_{r\downarrow0}\left[r\alpha_{1}(r)\frac{\phi(r)-\phi(0)}{r}+\alpha_{2}(r)\right]\\
    =&\frac{F_{\phi}(\phi(0),0)}{v'(0)}\frac{F_{r}(\phi(0),0)}{v_{r}(0)-F_{\phi}(\phi(0),0)}
    +\frac{F_{r}(\phi(0),0)}{v'(0)}\\
    =&\frac{F_{r}(\phi(0),0)}{v_{r}(0)-F_{\phi}(\phi(0),0)}.
  \end{aligned}\end{equation}
  Therefore, $\phi\in C^{1}[0,a]$ and \eqref{eq:FH0002} holds.

  \emph{Case 3.} $\gamma\in(-\infty,1)$ and $\phi\in Lip[0,a]$.
  In this case, $r\alpha_{1}(r)<1-\epsilon$ for $0<r<r_{\epsilon}$ for some $\epsilon>0$ and $r_{\epsilon}\in(0,a)$. Then
  $$0<G(r)<C_{\epsilon}r^{\epsilon-1},\hspace{2ex} r\in(0,a]$$
  for some $C_{\epsilon}>0$.
  Thus the right hand side of \eqref{eq:FH05} is integrable near $r=0$.
  Since $\phi\in Lip[0,a]$, we know $|\phi(r)-\phi(0)|\leq Cr$, $r\in[0,1]$ for some $C$. We know
  $$\lim_{r\rightarrow0}(\phi(r)-\phi(0))G(r)=0.$$
  We now integrate \eqref{eq:FH05} over $[0,r]$ to obtain \eqref{eq:FH06}. Thus, \eqref{eq:FH07a} and \eqref{eq:FH07b} is also obtained.
  Therefore, $\phi\in C^{1}[0,a]$ and \eqref{eq:FH0002} holds.
\end{proof}

\begin{rmk}
  If $\gamma\in(0,1)$, then $\phi\in C^{\gamma-\epsilon}[0,a]$ for any small $\epsilon>0$;
  If $F_{\phi}(\phi(0),0)\neq0$ and $v(r)=c r^{\beta}(1+o(1))$ for some $c\neq0$, $\beta>1$, then $\phi$ has a derivative at $r=0$.
\end{rmk}

\par\vspace*{2ex}\par\noindent
\textbf{Acknowledgments:}
The authors are very grateful to the professor Yuan Lou  for his helpful comments and Wenrui Hao for his useful suggestions.  Rui Li is sponsored by the China Scholarship Council. Rui Li also would like to thank the Department of Applied Computational Mathematics and Statistics of the University of Notre Dame for its hospitality when she was a visiting student.

\end{document}